\documentclass[reqno,a4paper]{amsart}
\usepackage{amsmath,amsfonts,mathrsfs,enumerate,graphicx,subfig}

\usepackage{xcolor,mathtools}
\definecolor{crimson}{rgb}{0.85, 0.08, 0.4}
\definecolor{bleudefrance}{rgb}{0.2, 0.5, 0.9}
\usepackage[pdftex,colorlinks,linkcolor=crimson,citecolor=bleudefrance,filecolor=black]{hyperref}

\newtheorem{theorem}{Theorem}[section]

\newtheorem{proposition}{Proposition}[section]
\newtheorem{corollary}[proposition]{Corollary}
\newtheorem{lemma}[proposition]{Lemma}

\newtheorem{definition}[proposition]{Definition}
\theoremstyle{remark}
\newtheorem{remark}[proposition]{Remark}

\usepackage{soul}

\numberwithin{equation}{section}

\DeclareMathOperator{\dimaff}{dim_{\mathsf{aff}}}
\DeclareMathOperator{\dimlyap}{dim_{\mathsf{Lyap}}}
\DeclareMathOperator{\dimh}{dim_{\mathsf{H}}}

\DeclareMathOperator{\GL}{GL}

\DeclareMathOperator{\OO}{O}

\DeclareMathOperator{\SL}{SL}

\DeclareMathOperator{\rk}{rank}
\DeclareMathOperator{\id}{id}

\DeclareMathOperator{\iiid}{id}
\DeclareMathOperator{\End}{End}
\DeclareMathOperator{\Mat}{Mat}

\DeclareMathOperator{\diam}{diam}

\newcommand{\triple}[1]{{\left\vert\kern-0.25ex\left\vert\kern-0.25ex\left\vert #1 
    \right\vert\kern-0.25ex\right\vert\kern-0.25ex\right\vert}}
\newcommand{\iii}{\mathtt{i}}
\newcommand{\I}{\mathcal{I}}

\newcommand{\threebar}[1]{{\left\vert\kern-0.25ex\left\vert\kern-0.25ex\left\vert #1 
    \right\vert\kern-0.25ex\right\vert\kern-0.25ex\right\vert}}
\newcommand{\J}{\mathcal{J}}
\newcommand{\jjj}{\mathtt{j}}
\newcommand{\kkk}{\mathtt{k}}

\newcommand{\lll}{\mathtt{l}}

\newcommand{\R}{\mathbb{R}}

\newcommand{\N}{\mathbb{N}}

\title[From self-affine measures to self-affine sets]{A variational principle relating self-affine measures to self-affine sets}
\author{Ian D. Morris and Cagri Sert} 
\email{i.morris@qmul.ac.uk }
\email{sertcagri@gmail.com}
\begin{document}

\begin{abstract}
A breakthrough result of B\'ar\'any, Hochman and Rapaport published in 2019 established that every self-affine measure on $\R^2$ satisfying certain mild non-degeneracy conditions has Hausdorff dimension equal to its Lyapunov dimension. In combination with a variational principle established earlier by Morris and Shmerkin this result implied as a corollary that the attractor of a planar affine iterated function system satisfying the same conditions necessarily has Hausdorff dimension equal to a value proposed by Falconer in 1988. In this article we extend the variational principle of Morris and Shmerkin from the planar context to arbitrary dimension. This allows a recent theorem of Rapaport on the dimensions of self-affine measures in $\R^3$ to be extended into a characterisation of the dimensions of the corresponding self-affine subsets of $\R^3$.  At the core of the present work is an algebraic result concerned with finding large Zariski-dense Schottky semigroups inside a given finitely generated completely reducible semigroup of linear transformations.
\end{abstract}
\maketitle

\section{Introduction and principal results}

An iterated function system (which we will usually abbreviate to IFS)  is conventionally defined to be any finite collection $(T_i)_{i \in \I}$ of contracting transformations of a complete metric space $X$. A now-classical theorem of J.E. Hutchinson (see \cite{Hu81}) states that if $(T_i)_{i \in \I}$ is an iterated function system acting on such a space $X$ then there exists a unique nonempty compact subset $Z$ of $X$ satisfying the equation $Z=\bigcup_{i\in \I} T_iZ$, and this set is usually called the \emph{attractor} of $(T_i)_{i \in \I}$. If additionally a probability vector $(p_i)_{i \in \I}$ is specified then there moreover exists a unique Borel probability measure $m$ on $X$ which satisfies $m=\sum_{i\in \I} p_i(T_i)_*m$. In this article we will always be concerned with iterated function systems which consist of invertible affine transformations of $\R^d$ and which contract a metric on $\R^d$ which is induced by a norm (which need not be the Euclidean norm). To avoid trivialities, we assume at all times that the cardinality of $\I$ is at least $2$.

We recall that an IFS $(T_i)_{i \in \I}$ acting on $\R^d$ is said to satisfy the \emph{open set condition} if there exists a nonempty open set $U \subseteq \R^d$ such that the images $T_iU$ for $i \in \I$ are pairwise disjoint subsets of $U$, and is said to satisfy the \emph{strong open set condition} if additionally the set $U$ can be chosen so as to have nonempty intersection with the attractor. In both cases, by replacing $U$ with its intersection with a large open ball we may freely assume that $U$ is bounded, and we will always make this assumption in the sequel. The IFS $(T_i)_{i \in \I}$ is further said to satisfy the \emph{strong separation condition} if there exists a nonempty compact set $X \subseteq \R^d$ such that the images $T_iX$ for $i \in \I$ are pairwise disjoint subsets of $X$.  If $(T_i)_{i \in \I}$ satisfies the open set condition and if each $T_i$ is a similarity transformation -- that is, if each $T_i$ has the form $T_ix\equiv r_iO_ix+v_i$ with $r_i \in (0,1)$, $O_i \in \OO(d)$ and $v_i \in \R^d$ -- it is classical that the Hausdorff and box dimensions of the attractor are both equal to the \emph{similarity dimension} of $(T_i)_{i \in \I}$, which is defined to be the unique $s \in (0,d]$ which satisfies $\sum_{i\in\I} r_i^s = 1$. Moreover, in this case the Hausdorff dimension of any self-similar measure $m=\sum_{i\in \I} p_i (T_i)_*m$ is precisely equal to $\sum_{i \in \I} p_i\log p_i / \sum_{i\in \I} p_i \log r_i$, which is at most the similarity dimension of $(T_i)_{i \in \I}$. Taking $p_i:=r_i^s$ where $s$ denotes the similarity dimension yields a self-similar measure whose Hausdorff dimension equals that of the attractor.

When the above hypotheses are relaxed so as to allow the maps $T_i$ to be invertible affine transformations the situation becomes much more delicate. Similarly to the above, the attractor of an affine iterated function system $(T_i)_{i \in \I}$ will be called a self-affine set, and a Borel probability measure on $\R^d$ satisfying an equation of the form $m=\sum_{i \in \I} p_i(T_i)_*m$ will be called a self-affine measure. Suppose that the affine transformations $T_i$ are each written in the form $T_ix\equiv A_ix+v_i$ with $A_i \in \GL_d(\R)$ and $v_i \in \R^d$, which we describe by  saying that each $A_i$ is the \emph{linearisation} of the corresponding affine map $T_i$. We recall that the \emph{singular values} of $A \in \GL_d(\R)$ are the positive square roots of the eigenvalues of the positive definite linear map $A^\top A$, and are conventionally denoted $\sigma_1(A),\ldots,\sigma_d(A)$ in decreasing order, with each value being repeated if necessary in the case of a multiple eigenvalue. The \emph{singular value function} described by K. Falconer in \cite{Fa88} (though for some antecedent ideas see \cite{douady-oesterle,kaplan-yorke}) is defined by
\[\varphi^s(A):=\left\{\begin{array}{cl}\sigma_1(A) \cdots \sigma_{\lfloor s\rfloor}(A) \sigma_{\lceil s\rceil}(A)^{s-\lfloor s\rfloor}&\text{if }0 \leq s \leq d,\\ |\det A|^{\frac{s}{d}}&\text{if }d \geq s,\end{array}\right.\] 
where $A \in \GL_d(\R)$. The \emph{affinity dimension} of $(T_i)_{i \in \I}$, denoted $\dimaff (T_i)_{i \in \I}$, may then be defined to be the unique $s \geq 0$ such that
\[\lim_{n \to \infty} \frac{1}{n} \log\sum_{i_1,\ldots,i_n \in \I} \varphi^s(A_{i_1}\cdots A_{i_n})=0,\]
and the \emph{Lyapunov dimension} of a self-affine measure $m=\sum_{i \in \I} p_i(T_i)_*m$ is defined to be the unique non-negative solution $t$ to the equation
\[\lim_{n \to \infty} \frac{1}{n} \sum_{i_1,\ldots,i_n \in \I} p_{i_1}\cdots p_{i_n}\log \varphi^t(A_{i_1}\cdots A_{i_n})=\sum_{i \in \I} p_i\log p_i. \]
When the transformations $T_i$ are all similitudes the affinity dimension is easily seen to be equal to the similarity dimension defined previously. The Lyapunov dimension may also be defined for a more general class of measures induced by an affine IFS, but we defer this more general definition to the following section.

Since the publication of Falconer's article in 1988, a long-standing programme of research has been dedicated to finding explicit conditions on an affine IFS which guarantee that its attractor has Hausdorff dimension equal to the affinity dimension and that its self-affine measures all have Hausdorff dimension equal to their Lyapunov dimensions. It has long been known that these conclusions do not follow from the open set condition alone, nor even from the strong separation condition \cite{Be84,Mc84}, but it was established in \cite{Fa88} that if a sufficiently strongly contracting tuple of linear maps $(A_i)_{i\in \I}\in \GL_d(\R)^\I$ is specified, then for Lebesgue a.e. $(v_i)_{i \in \I} \in (\R^d)^\I$ the attractor of the IFS $(T_i)_{i\in \I}$ defined by $T_ix:=A_ix+v_i$ has Hausdorff dimension equal to its affinity dimension. (A corresponding result for self-affine measures followed rather later in \cite{JoPoSi07}.) On the other hand this result gave no \emph{explicit} examples of IFS for which the Hausdorff dimension of the attractor equals the affinity dimension, nor any verifiable sufficient condition for this property to hold.

Let us say that $(A_i)_{i \in \I}$ is \emph{irreducible} if there does not exist a nonzero proper vector subspace $V$ of $\R^d$ which is preserved by every $A_i$, and \emph{strongly irreducible} if no finite union of such subspaces is preserved by every $A_i$. For the purposes of this introduction we will also call $(A_i)_{i \in \I}$ \emph{proximal} if the semigroup generated by $\{A_i \colon i \in \I\}$ contains a linear map whose first- and second-largest eigenvalues have distinct absolute values. (The notion of proximality will be revisited in greater depth in \S\ref{sec.preliminaries} below.) We call an IFS $(T_i)_{i \in \I}$ irreducible, strongly irreducible or proximal if its vector of linearisations $(A_i)_{i \in \I}$ has the corresponding property. Building on a substantial body of research subsequent to \cite{Fa88}, which we do not attempt to survey here, the following breakthrough result of B. B\'ar\'any, M. Hochman and A. Rapaport was published in 2019 in \cite{BaHoRa19}:
\begin{theorem}[B\'ar\'any-Hochman-Rapaport]\label{th:bahora}
Let $(T_i)_{i \in \I}$ be a proximal and strongly irreducible affine iterated function system acting on $\R^2$ which satisfies the strong open set condition. Then:
\begin{enumerate}[(i)]
\item\label{it:bahora-one}
If $(p_i)_{i \in \I}$ is a non-degenerate probability vector then the self-affine measure $m=\sum_{i \in \I} p_i(T_i)_*m$ has Hausdorff dimension equal to its Lyapunov dimension.
\item\label{it:bahora-two}
The Hausdorff dimension of the attractor of $(T_i)_{i \in \I}$ is precisely $\dimaff (T_i)_{i \in \I}$.\end{enumerate}
\end{theorem}
The hypothesis of proximality can in fact be freely removed from Theorem \ref{th:bahora}, since when proximality is absent the remaining hypotheses imply that the linear maps $|\det A_i|^{-1/2}A_i$ preserve an inner product on $\R^2$. It then follows that up to a simultaneous change of basis for $\R^2$ the IFS $(T_i)_{i \in \I}$ consists of similarity transformations, so the conclusions \eqref{it:bahora-one}--\eqref{it:bahora-two} can be obtained in the non-proximal case from the classical theorem of Hutchinson.

The conclusions of Theorem \ref{th:bahora} can fail to hold when $(T_i)_{i \in \I}$ fails to be irreducible (see \cite{Be84,Mc84}) or when $(T_i)_{i \in \I}$ is irreducible but not strongly irreducible (see \cite{Fr12}). If the strong open set condition is weakened to the open set condition then it becomes possible for the attractor to be a singleton set (as is the case for an example noted by G.A. Edgar in \cite{Ed92}) and in this case the conclusions of Theorem \ref{th:bahora} clearly also do not hold. On the other hand the strong open set condition may nonetheless be weakened considerably to an \emph{exponential separation} condition which stipulates that two maps $T_{i_1}\cdots T_{i_n}$ and $T_{j_1}\cdots T_{j_n}$ must in a certain precise sense not be too similar to one another when $n$ is fixed and the vectors $(i_1,\ldots,i_n), (j_1,\ldots,j_n) \in \I^n$ are distinct, if it is additionally assumed that the contractions $T_i$ do not share a common fixed point (see \cite{HoRa22}).

A Ledrappier-Young formula established by B\'ar\'any and K\"aenm\"aki in \cite{BaKa17} relates the Hausdorff dimensions of self-affine measures on $\R^2$ to the Hausdorff dimensions of their projections onto one-dimensional subspaces, and as a consequence the proof of Theorem \ref{th:bahora}\eqref{it:bahora-one} is principally concerned with the estimation of the dimensions of these projections in ``typical'' directions. The Ledrappier-Young formula of B\'ar\'any and K\"aenm\"aki has since been extended to higher dimensions by D.-J. Feng \cite{Fe19}. An extension of Theorem \ref{th:bahora}\eqref{it:bahora-one} to affine IFS acting on $\R^3$ was recently obtained by A. Rapaport in \cite{Ra22} as follows:
\begin{theorem}[Rapaport]\label{th:rapa}
Let $(T_i)_{i \in \I}$ be a proximal and strongly irreducible affine iterated function system acting on $\R^3$. Then for every non-degenerate probability vector $(p_i)_{i \in \I}$, the self-affine measure $m=\sum_{i \in \I} p_i(T_i)_*m$ has Hausdorff dimension equal to its Lyapunov dimension.
\end{theorem}
 It was shown by Falconer in \cite{Fa88} that the upper box dimension of the attractor of an affine IFS $(T_i)_{i\in \I}$ is always less than or equal to $\dimaff(T_i)_{i \in \I}$, and this implies in particular that the Hausdorff dimension of the attractor must be less than or equal to the affinity dimension. The difficulty in obtaining results along the lines of Theorem \ref{th:bahora}\eqref{it:bahora-two} therefore lies in bounding the dimension of the attractor from below. The present work is concerned with a mechanism by which results concerning the dimensions of self-affine measures -- such as Theorem \ref{th:bahora}\eqref{it:bahora-one} and Theorem \ref{th:rapa} above -- may be converted into sharp lower bounds on the dimensions of self-affine sets, such as Theorem \ref{th:bahora}\eqref{it:bahora-two}. 

When $(T_i)_{i \in \I}$ consists of similitudes there always exists a self-affine measure $m=\sum_{i \in \I}p_i (T_i)_*m$ which corresponds to a non-degenerate probability vector and has Lyapunov dimension equal to the affinity dimension. Since every such measure has support equal to the attractor of the IFS, this leads directly to the desired lower bound in the classical self-similar case of Hutchinson's article \cite{Hu81}. For irreducible affine IFS which do \emph{not} consist of similitudes a self-affine measure with the desired Lyapunov dimension cannot exist, and indeed the supremum of the Lyapunov dimensions of the possible self-affine measures is necessarily strictly smaller than the affinity dimension \cite{MoSe19}. Theorem \ref{th:bahora}\eqref{it:bahora-one} may therefore not be applied directly to deduce \eqref{it:bahora-two}, but is applied indirectly via an additional result as follows.

Given a finite set $\I$, let us say that a \emph{word} of length $n$ over $\I$ is a sequence of symbols $i_1i_2\cdots i_n$ such that every $i_j$ is an element of $\I$. We let $\I^n$ denote the set of all words of length $n$ over $\I$ and write $\I^*$ for the set of all words over $\I$ having arbitrary nonzero length. We denote the length of the word $\iii \in \I$ by $|\iii|$. If $\iii=i_1\cdots i_n$ and $\jjj=j_1\cdots j_m$ are words over $\I$ then we define their concatenation $\iii\jjj$ to be the word of length $|\iii|+|\jjj|$ defined by $\iii\jjj:=i_1i_2\cdots i_nj_1j_2\cdots j_m$. The operation of concatenation gives $\I^*$ the structure of a semigroup. If $(T_i)_{i \in \I}$ is an affine IFS then for every $\iii=i_1\cdots i_n\in \I^*$ we define $T_\iii:=T_{i_1}T_{i_2}\cdots T_{i_n}$, and similarly if $(A_i)_{i \in \I} \in \GL(V)^\I$ for some finite-dimensional real vector space $V$ then for every $\iii=i_1\cdots i_n\in \I^*$ we define $A_\iii:=A_{i_1}A_{i_2}\cdots A_{i_n}$. If $(T_i)_{i \in \I}$ is an IFS then for every $n \geq 1$ the affine IFS $(T_\iii)_{\iii \in \I^n}$ has the same attractor as $(T_i)_{i\in \I}$ and additionally satisfies $\dimaff (T_\iii)_{\iii \in \I^n}=\dimaff (T_i)_{i \in \I}$. We observe that strong irreducibility of $(T_\iii)_{\iii \in \I^n}$ is equivalent to strong irreducibility of $(T_i)_{i\in \I}$ and that the same equivalence holds for proximality. On the other hand irreducibility of $(T_i)_{i \in \I}$ does not in general imply irreducibility of $(T_\iii)_{\iii \in \I^n}$.

The principle applied in \cite{BaHoRa19} to deduce \eqref{it:bahora-two} from \eqref{it:bahora-one} is the following. Since the Lyapunov dimension of a self-affine measure $m=\sum_{i \in \I} p_i(T_i)_*m$ cannot approach arbitrarily closely to the affinity dimension of $(T_i)_{i \in \I}$, one instead may look for self-affine measures of the form $m=\sum_{\iii \in \I^n} p_\iii (T_\iii)_*m$ with Lyapunov dimension close to $\dimaff (T_i)_{i \in \I}$, where $n\geq 1$ is large and $(p_\iii)_{\iii \in \I^n}$ is an arbitrary probability vector. Unfortunately this too may be impossible if the new probability vector $(p_\iii)_{\iii \in \I^n}$ is required to be non-degenerate (see \cite[\S3.2]{MoSh19}). This problem can be evaded by allowing $(p_\iii)_{\iii \in \I^n}$ to be degenerate, but now Theorem \ref{th:bahora}\eqref{it:bahora-one} is not directly applicable. This obstacle may in turn be avoided by defining $\J\subset \I^n$ to be the set of all words $\iii \in \I^n$ such that $p_\iii$ is nonzero, and viewing the same self-affine measure $m=\sum_{\iii \in \I^n} p_\iii (T_\iii)_*m=\sum_{\jjj \in \J} p_\jjj (T_\jjj)_*m$ as corresponding to a non-degenerate probability vector $(p_\jjj)_{\jjj \in \J}$ and an \emph{a priori} smaller affine IFS $(T_\jjj)_{\jjj \in \J}$; but this introduces the problem that unlike the IFS $(T_\iii)_{\iii \in \I^n}$, the IFS $(T_\jjj)_{\jjj \in \J}$ might fail to inherit the proximality and strong irreducibility properties of $(T_i)_{i\in \I}$ which are required in order to make Theorem \ref{th:bahora}\eqref{it:bahora-one} function. An entirely new theorem is therefore required in order to ensure that all of these constraints can be met simultaneously. In the planar case the required result is as follows:
\begin{theorem}[Morris-Shmerkin \cite{MoSh19}]\label{th:mosh}
Let $(T_i)_{i \in \I}$ be a proximal and irreducible affine IFS acting on $\R^2$ and satisfying $\dimaff (T_i)_{i \in \I} \in (0,2)$. Then for every $\delta>0$ there exist $n\geq 1$ and $\J \subseteq \I^n$ such that:
\begin{enumerate}[(i)]
\item
The self-affine measure $m=\frac{1}{\#\J}\sum_{\jjj \in \J} (T_\jjj)_*m$ has Lyapunov dimension greater than or equal to $\dimaff (T_i)_{i \in \I}-\delta$.
\item
There exists a cone $\mathcal{K}$ in $\R^2$ which is strictly preserved by all of the linear maps $A_\jjj$ such that $\jjj \in \mathcal{J}$.
\item\label{it:sir}
If $(A_i)_{i \in \I}$ is strongly irreducible then so too is $(A_\jjj)_{\jjj \in \mathcal{J}}$.
\item
If $(T_i)_{i \in \I}$ satisfies the strong open set condition, then $(T_\jjj)_{\jjj \in \mathcal{J}}$ satisfies the strong separation condition.
\end{enumerate}
\end{theorem}
Theorem \ref{th:bahora}\eqref{it:bahora-two} holds trivially if the condition $\dimaff (T_i)_{i \in \I}<2$ is not met, since in this case the attractor is easily shown to equal the closure of the open set considered in the strong open set condition (see e.g. \cite[Lemma 5.4]{MoSh19}). Otherwise, since the attractor $Z'$ of $(T_\jjj)_{\jjj \in \J}$ is necessarily a subset of the attractor $Z$ of $(T_i)_{i\in \I}$, by applying Theorem \ref{th:bahora}\eqref{it:bahora-one} to the self-affine measure $m=\frac{1}{\#\J}\sum_{\jjj \in \J} (T_\jjj)_*m$ whose support equals $Z'$, the lower bound $\dimh Z \geq \dimaff (T_i)_{i\in\I}-\delta$ follows directly. 

For the purposes of Theorem \ref{th:mosh} a \emph{cone} is a closed, convex, positively homogenous set $\mathcal{K}\subset \R^2$ with nonempty interior and having the property that $\mathcal{K} \cap -\mathcal{K}=\{0\}$. A linear map $A \colon \R^2 \to \R^2$ is here said to \emph{strictly preserve} a cone $\mathcal{K}$ if $A(\mathcal{K}\setminus\{0\})$ is a subset of the topological interior of $\mathcal{K}$. This hypothesis was included in the work \cite{MoSh19} for compatibility with the earlier works \cite{Ba15,FaKe18} which applied only to planar affine IFS whose linearisations strictly preserve a cone.

In order to state the main result of the present article, a few further definitions are required. We will say that $(A_i)_{i \in \I} \in \GL_d(\R)^\I$ is \emph{completely reducible} if $\R^d$ can be written as a direct sum $\R^d=\bigoplus_{j=1}^k V_j$ where each $V_j$ is a vector subspace of $\R^d$ such that the restriction of $(A_i)_{i \in \I}$ to $V_j$ is irreducible. In particular if $(A_i)_{i \in \I}$ is irreducible then it is completely reducible. The \emph{Zariski closure} of a subsemigroup $\Gamma$ of $\GL_d(\R)$ is by definition the smallest algebraic variety in $\GL_d(\R)$ which contains $\Gamma$, and is always a Lie group with finitely many connected components. If $G$ is the Zariski closure of a subgroup $\Gamma$ of $\GL_d(\R)$ then we write $G_c$ for the unique connected component of $G$ (with respect to the Zariski topology) which contains the identity. Given $(A_i)_{i \in \I} \in \GL_d(\R)^\I$ and $k \in \{1,\ldots,d-1\}$ we will say that $(A_i)_{i \in \I}$ is $k$-irreducible, $k$-strongly irreducible or $k$-proximal if the vector of linear maps $(A_i^{\wedge k})_{i \in \I} \in \GL(\wedge^k \R^d)^\I$ is correspondingly irreducible, strongly irreducible or proximal. Lastly, following \cite{bochi-gourmelon} we say that $(A_i)_{i\in \I}$ is $k$-dominated if the sequence
\[\max_{\iii \in \I^n} \frac{\sigma_{k+1}(A_\iii)}{\sigma_k(A_\iii)}\]
converges to zero exponentially as $n\to \infty$. We call an affine IFS $(T_i)_{i \in \I}$ completely reducible if its vector of linearisations is completely reducible.

The main result of this article is the following:
\begin{theorem}\label{th:main}
Let $(T_i)_{i \in \I}$ be a completely reducible affine IFS acting on $\R^d$ and satisfying $0<\dimaff (T_i)_{i \in \I}<d$, and let $G\leq \GL_d(\R)$ denote the Zariski closure of the semigroup $\{A_\iii \colon \iii \in \I^*\}$. Then for every $\delta>0$ there exist an integer $n \geq 1$ and a nonempty set $\J \subseteq \I^n$ such that the following properties hold:
\begin{enumerate}[(i)]
    \item \label{it:uniform}
    The Lyapunov dimension of the self-affine measure $m=\frac{1}{\#\J}\sum_{\jjj \in \J} (T_\jjj)_*m$ is at least $\dimaff(T_i)_{i \in \I}-\delta$.
    \item\label{it:zar}
    The Zariski closure of the semigroup $\{A_\lll \colon \lll \in \J^*\}$ is precisely $G_c$.
    \item\label{it:gaps}
    For every integer $k \in \{1,\ldots,d-1\}$ such that $(A_i)_{i\in \I}$ is both $k$-proximal and $k$-strongly irreducible, $(A_\jjj)_{\jjj \in \J}$ is both $k$-dominated and $k$-strongly irreducible.
    \item\label{it:strongsep}
    If $(T_i)_{i \in \I}$ satisfies the strong open set condition, then $(T_\jjj)_{\jjj \in \J}$ satisfies the strong separation condition.
\end{enumerate}
\end{theorem}
As was previously remarked by Rapaport in \cite{Ra22}, the combination of Theorem \ref{th:main} with Theorem \ref{th:rapa} directly implies:
\begin{theorem}\label{th:rapco}
Let $(T_i)_{i \in \I}$ be a proximal and strongly irreducible affine IFS acting on $\R^3$ which satisfies the strong open set condition. Then the Hausdorff dimension of the attractor of $(T_i)_{i\in \I}$ is equal to $\dimaff (T_i)_{i \in \I}$. 
\end{theorem}
The proof of Theorem \ref{th:mosh} adapted earlier arguments used by D.-J. Feng and P. Shmerkin in their proof of the continuity of the affinity dimension in \cite{FeSh14} in combination with arguments to show that strong irreducibility of $(A_i)_{i \in \I}$ can be retained by $(A_\jjj)_{\jjj \in \J}$. This part of the proof took advantage of relatively simple and essentially analytic characterisations of strong irreducibility which are available in two dimensions. The much broader scope of Theorem \ref{th:main}\eqref{it:zar} relative to Theorem \ref{th:mosh}\eqref{it:sir} means that the algebraic arguments involved in the proof of Theorem \ref{th:main} are correspondingly more substantial. 

The remainder of this article is structured as follows. In \S\ref{sec.preliminaries} we review those facts from linear algebra, the theory of reductive linear algebraic groups, and the thermodynamic formalism of affine IFS which will be required in Theorem \ref{th:main}. Here we will in particular revisit in greater breadth and detail some of the concepts appealed to already in the introduction. In \S\ref{sec.core} we prove an algebraic result, Theorem \ref{thm.key.intrinsic}, which underpins the proof of Theorem \ref{th:main}. The proof of Theorem \ref{th:main} is presented in \S\ref{sec.pf.main.thm}.

%
%

\section{Preliminaries}\label{sec.preliminaries}

\subsection{Linear algebra}

We begin by describing in more detail the dynamics of projective linear maps, some quantitative notions of proximality, and the notions of proximality index and domination for sets of linear transformations. 

\subsubsection{Dynamics of projective linear maps}\label{subsub.projective}
 
Let $V$ be a finite dimensional Euclidean vector space and $\mathbf{P}(V)$ the corresponding projective space. Where no confusion results, we will use the same symbol $x$ to denote a nonzero vector $x \in V$ and the one-dimensional vector space generated by $x$, which is an element of $\mathbf{P}(V)$. The exterior products $\wedge^k V$ are endowed with the  Euclidean structure induced by that of $V$. We equip $\mathbf{P}(V)$ with the metric $d(x,y):=\sin \angle (x,y)=\frac{\|x \wedge y\|}{\|x\|.\|y\|}$. Given a nonempty set $K\subset \mathbf{P}(V)$ and an element $x \in \mathbf{P}(V)$ we let $d(x,K)$ denote the infimal distance of $x$ to an element of $K$. Given $g \in \End(V)$ we write $\lambda_1(g)$ for the spectral radius of $g$, and for $k=1,\ldots,\dim V$ we write $\sigma_k(g)$ for the $k^{th}$-singular value of $g$. As indicated in the introduction we label singular values in decreasing order as $\sigma_1(g) \geq \ldots \geq \sigma_d(g) \geq 0$. We recall the following definition (see \cite{tits.free,  benoist.linear1, breuillard-gelander.annals, abels.proximal}):
 
\begin{definition}[Proximal map] A linear map $g \in \End(V)$ is said to be proximal if it has a unique eigenvalue of maximal modulus $\lambda_1(g)$. For a proximal map $g$, we denote by $v_{g}^{+}$, the eigenline in $\mathbf{P}(V)$ corresponding to the maximal-modulus eigenvalue of $g$ and $H_{g}^{<}$ the complementary $g$-invariant hyperplane, which we will frequently identify with the corresponding compact subset of $\mathbf{P}(V)$.
\end{definition}

The following is a by-now classical quantification of the previous definition going back to the work of Abels, Margulis and Soifer in \cite{AMS} (though see also \cite{benoist.linear1, breuillard-gelander.annals}). Given $\varepsilon>0$ and a hyperplane $H \subset \mathbf{P}(V)$ we write $B_H^\varepsilon:=\{x \in \mathbf{P}(V) \colon d(x, H) \geq \varepsilon\}$ and for $x \in \mathbf{P}(V)$ we write $b_x^{\varepsilon}:=\{y \in \mathbf{P}(V) \colon d(x,y) < \varepsilon\}$. When $g \in \End(V)$ is a proximal element, we write $b_g^\varepsilon:=b_{x_g^+}^\varepsilon$ and $B_g^\varepsilon:=B_{H_g^<}^\varepsilon$.

\begin{definition}[$(r,\varepsilon)$-proximal map]
Let $0 < \varepsilon \leq r$. A proximal element $ g \in \End(V) $ is said to be $(r, \varepsilon)$-proximal if $d(v_{g}^{+},H_{g}^{<})\geq 2r$ and if $g_{|B_g^\varepsilon}$ is $\varepsilon$-Lipschitz.
\end{definition}

The following observation (due to Benoist in \cite{benoist.linear1}) provides fine control over the linear action of an $(r,\varepsilon)$-proximal transformation on a vector $x$ in terms of its spectral radius and the projective configuration of $x, v_g^+$ and $H_g^<$. We borrow the precise formulation from \cite[Lemma 2.7]{breuillard-sert} and refer to that article for a proof.

\begin{lemma}\label{lemma.radius.vs.norm} For every $0 < \varepsilon \leq r$  there exists a constant $D_{r,\varepsilon}>1$ such that for every $(r,\varepsilon)$-proximal endomorphism $g$ of $V$ and every $x \in V\setminus\{0\}$ with $d(x,H_g^{<})\geq r$,
\begin{equation*}
D_{r,\varepsilon}^{-1} \leq \frac{\|gx\|}{\|x\|} \frac{d(v_g^+,H_g^{<})}{d(x,H_g^{<})} \frac{1}{\lambda_1(g)} \leq D_{r,\varepsilon},
\end{equation*}
and these constants may be chosen so as to satisfy $\underset{\varepsilon \rightarrow 0}{\lim} \, D_{r,\varepsilon} = 1$ for every fixed $r>0$.

\end{lemma}

The following result expresses the principle that the spectral radius is nearly multiplicative on any semigroup of $(r,\varepsilon)$-proximal elements which satisfy a certain additional compatibility condition. It is likewise due to Benoist, this time in \cite[Prop.\ 1.4]{benoist.linear2}; we borrow the precise formulation from \cite[Proposition 2.7]{breuillard-sert} and refer to the latter for a proof.

\begin{proposition}\label{prop.products.proximals}
For every $0<\varepsilon \leq r$ there exists a positive constant $D_{r,\varepsilon}>1$ with the following properties. If $g_{1}, \ldots g_{\ell}$ are $(r,\varepsilon)$-proximal linear transformations of $V$ satisfying $d(v_{g_{\ell}}^{+},H_{g_{1}}^{<}) \geq 6r $ and satisfying $d(v_{g_{j}}^{+},H_{g_{j+1}}^{<}) \geq 6r $ for all $j=1, \ldots \ell-1$, we have that for all integers $n_{1}, \ldots, n_{l} \geq 1$, the linear transformation $g_{\ell}^{n_{\ell}}\ldots g_{1}^{n_{1}}$
is $(2r,2\varepsilon)$-proximal and satisfies
\begin{equation*}
\beta(g_1,\ldots,g_\ell)  D_{r,\varepsilon}^{-\ell} \leq \frac{\lambda_{1}(g_{\ell}^{n_{\ell}}\ldots g_{1}^{n_{1}})}{\lambda_{1}(g_{\ell})^{n_{\ell}} \ldots \lambda_{1}(g_{1})^{n_{1}}} \leq D_{r,\varepsilon}^{\ell} \beta(g_1,\ldots,g_\ell)
\end{equation*}
where 
$$\beta(g_1,\ldots,g_\ell)  := \frac{d(v_{g_\ell}^+,H_{g_1}^{<}) d(v_{g_1}^+,H_{g_2}^{<})\ldots d(v_{g_{\ell-1}}^+,H_{g_\ell}^{<})}{d(v_{g_1}^+,H_{g_1}^{<})\ldots d(v_{g_\ell}^+,H_{g_\ell}^{<})}.$$
Moreover we have $d(v^+_{g_{\ell}^{n_{\ell}}\ldots g_{1}^{n_{1}}},v^+_{g_\ell}) \leq \varepsilon$ and $d_H(H_{g_{\ell}^{n_{\ell}}\ldots g_{1}^{n_{1}}}^<,H_{g_1}^<)\leq \varepsilon$ where $d_H$ denotes the Hausdorff distance on compact subsets of $\mathbf{P}(V)$. Finally, these constants may be chosen so as to satisfy the condition $\lim_{\varepsilon \rightarrow 0}D_{r, \varepsilon} =1$ for every $r>0$.
\end{proposition}

We introduce the following terminology borrowed from \cite[Def. 1.7]{benoist.linear2} and motivated by the previous proposition:

\begin{definition}[Schottky family] \label{defnarepsSchottky1}
 A subset $E$ of $\GL(V)$ is called an $(r,\varepsilon)$-Schottky family if both:
 \begin{enumerate}[(i)]
\item
For all $\gamma \in E$, $\gamma$ is $(r,\varepsilon)$-proximal, and\item
$d(v_{\gamma}^{+},H_{\gamma'}^{<}) \geq 6r$, for all $\gamma, \gamma' \in E$.
\end{enumerate}
\end{definition}

A direct consequence of Proposition \ref{prop.products.proximals} in relation to $(r,\varepsilon)$-Schottky families is the following.

\begin{corollary}\label{corol.family.to.semigroup.onerep}
Let $E\subseteq \GL(V)$ be an $(r,\varepsilon)$-Schottky family with $r > 4\varepsilon$. Then the semigroup generated by $E$ is an $(r/2,2\varepsilon)$-Schottky family. \qed
\end{corollary}

\subsubsection{Proximality index}

Let $V$ be a finite dimensional real vector space as before and $\Gamma$ a subsemigroup of $\GL(V)$. Viewing $\Gamma$ in the real algebra $\End(V)$, let $\mathbb{R}\Gamma$ be the real algebra generated by $\Gamma$ in $\End(V)$ and $\overline{\R\Gamma}$ its closure in the usual topology of $\End(V)$. We define the proximality index of $\Gamma$ in $V$ as
$$
r_{\Gamma,V}:=\min\{\rk \pi : \pi \in \overline{\R \Gamma} \setminus \{0\}\}.
$$

By a useful result of Goldsheid--Margulis \cite[Theorem 3.6]{goldsheid-margulis} (see also \cite[Lemma 6.23]{BQ.book}), the proximality index depends on the Zariski closure of $\Gamma$. Here and throughout we use the notation $\overline{X}^Z$ to denote the Zariski closure of a set $X\subset \GL(V)$.
\begin{theorem}[Goldsheid--Margulis]\label{th:gm}
Let $V$ be a finite dimensional real vector space and $\Gamma$ a completely reducible subsemigroup of $\GL(V)$. Let $G<\GL(V)$ be the Zariski closure $\overline{\Gamma}^Z$ of $\Gamma$ in $\GL(V)$. Then, 
$$
r_{\Gamma,V}=r_{G,V}.
$$
\end{theorem}

It is easy to see that if $H_0$ is a finite-index subgroup of a group $H<\GL(V)$, then $r_{H_0,V}=r_{H,V}$ (e.g.~ by using \cite[Lemma A.2]{bochi-sambarino-potrie}).

\subsubsection{Dominated sets of linear transformations}

Here we recall the notion of domination for a bounded set of linear transformations as studied by \cite{yoccoz, avila-bochi-yoccoz, bochi-gourmelon}. We also record a characterisation of dominated sets via cones due to \cite{bochi-gourmelon}. For a set $S \subseteq \GL_d(\R)$ and $n \in \N$, we denote by $S^n$ the set of $n$-fold products of elements of $S$, namely $S^n:=\{g_1 \ldots g_n \colon g_i \in S\}$.

\begin{definition}[Dominated family]
For $k=1,\ldots,d-1$, a relatively compact subset $S \subset \GL_d(\R)$ is said to be $k$-dominated if there is $\varepsilon>0$ such that for every large enough $n \in \N$ and every $g \in S^n$,
$$\frac{\sigma_{k+1}(g)}{\sigma_k(g)} \leq (1-\varepsilon)^n.$$
\end{definition}
Clearly, a set $S \subseteq \GL(\R^d)$ is $k$-dominated if and only if $\wedge^k(S) \in \GL(\wedge^k(\R^d))$ is $1$-dominated.

For $k=1,\ldots,d$, we denote by $\mathbf{Gr}_k(\R^d)$ the Grassmannian of $k$-dimensional subspaces of $\R^d$. An important result of Bochi and Gourmelon \cite{bochi-gourmelon} (see also  \cite{bochi-sambarino-potrie}) provides an alternative characterisation of $k$-dominated sets via a cone-type condition. This is expressed in the following.
\begin{theorem}\cite[Theorem B]{bochi-gourmelon}
A relatively compact set $S$ of $\GL_d(\R)$ is $k$-dominated if and only if there exist a $(d-k)$-dimensional subspace $W\subset \R^d$ and a non-empty subset $C$ of $\mathbf{Gr}_k(\R^d)$ such that $\overline{\bigcup_{g \in S} gC}$ is a compact subset of the interior of $C$ and such that all elements of $C$ are transverse to $W$.
\end{theorem} 
Here, strictly invariant means that the closure of $C$ is mapped to the interior of $C$ by every element of $S$.


\subsection{Linear reductive groups} \label{subsec.reductive} 
Here we include a brief overview of some aspects of reductive linear algebraic groups that we will use in the sequel. Our interest  in this class of groups arises from the fact that they occur as the Zariski closures of subsemigroups of $\GL_d(\mathbb{R})$ which act completely reducibly on $\mathbb{R}^d$. For a more detailed exposition of the theory of reductive linear algebraic groups, we refer the reader to \cite{BQ.book,borel-tits,knapp,mostow.book}.

\subsubsection{Definitions and complete reducibility}

A linear real algebraic subgroup $G$ of $\GL_d(\R)$ is said to be real reductive if it has no non-trivial normal subgroup consisting of unipotent matrices. It is well-known (\cite[Ch.4]{chevalley}) that the action on $\mathbb{R}^d$ of a real reductive group $G$ is completely reducible. As is also well-known, conversely, if $\Gamma$ is a subsemigroup of $\GL_d(\mathbb{R})$ that acts completely reducibly on $\mathbb{R}^d$ then the Zariski closure of $\Gamma$  is a real reductive group\footnote{Indeed, if it is not real reductive, then it contains a non-trivial normal subgroup $N$ consisting of unipotent matrices. Let $V_1$ be a $H$-irreducible subspace of $\R^d$ on which $N$ acts non-trivially. By a classical result of Kolchin, the subspace $V_0$ of fixed vectors of $N$ in $V_1$ is a non-trivial proper subspace of $V_1$. Since $N$ is normal in $H$, $V_0$ is invariant under $G$, contradicting irreducibility of the $H$-action on $V_1$.}.

Given a real reductive group $G$, let $\mathbf{G}$ denote the underlying affine real algebraic group so that $G$ can be identified with the group $\mathbf{G}(\R)$ of real points of $\mathbf{G}$. We will denote by $G_c$ the Zariski connected component $\mathbf{G}^\circ(\R)$ of $G$ which contains the identity element. 
Let $\mathbf{A}$ be a maximal $\R$-split torus of $\mathbf{G}$, $\Sigma$ a root system for the pair $(\mathbf{G}, \mathbf{A})$, $\Sigma^+$ a choice of positive roots, and $A$ the group of real points of $\mathbf{A}$. 

Let $\mathfrak{g}$ and $\mathfrak{a}$ be, respectively, the Lie algebras of $G$ and $A$. The Lie algebra $\mathfrak{a}$ is the direct sum \begin{equation}\label{torus}\mathfrak{a}=\mathfrak{a}_Z\oplus \mathfrak{a}_S\end{equation} of the Lie subalgebra $\mathfrak{a}_Z$ of $Z(G_c)\cap A$, where $Z(G_c)$ is the center of $G_c$, and the Lie subalgebra $\mathfrak{a}_S$ of $A \cap [G,G]$, where $[G,G]$ is the (closed) commutator subgroup of $G$, which is a semisimple Lie subgroup. 

For a character $\hat{\alpha}: A \to \R^\ast_+$, we will denote by $\alpha$ its differential which is an element of the dual $\mathfrak{a}^\ast$. The set $\Sigma$ constitutes a root system in $\mathfrak{a}^\ast$ and we denote 
$$
\mathfrak{a}^+:=\{x \in \mathfrak{a} \colon \alpha(x) \geq 0 \; \text{for every} \; \alpha \in \Sigma^+ \}
$$
the Weyl chamber corresponding to the choice of positive roots of $\mathfrak{a}$. 
We fix a scalar product $\langle \cdot, \cdot \rangle$ on $\mathfrak{a}$ whose restriction to $\mathfrak{a}_S$ is given by the Killing form, and which is defined arbitrarily on $\mathfrak{a}_Z$, and for which the direct sum in \eqref{torus} is orthogonal. We denote by $\mathfrak{a}^{++}$ the interior of the Weyl chamber in the Euclidean space $\mathfrak{a}$. Let $\Pi=\{\alpha_1,\ldots,\alpha_{d_S}\}$ be the simple roots in $\Sigma^+$, where $d_S$ is the semisimple rank of $G$, i.e.~ $d_S=\dim \mathfrak{a}_S$. 

To every simple root $\alpha \in \Pi$, one can associate a non-zero dominant weight $\omega_\alpha \in \mathfrak{a}^\ast$ such that  $\omega_\alpha(\mathfrak{a}_Z)=0$ and $2\frac{\langle\omega_{\alpha},\beta\rangle}{\langle \beta,\beta \rangle}=\delta_{\alpha,\beta}$ if $\beta \in \Pi$. The weights $\omega_\alpha$ are called fundamental weights of $G$. Every non-negative integer linear combination of the fundamental weights $\omega_{\alpha}$ is the highest weight of some irreducible representation of $G_c$.



\subsubsection{Cartan projection} Let $K_c$ be a maximal compact subgroup of $G_c$ and $K$ be the compact maximal subgroup of $G$ containing $K_c$. We have the Cartan decomposition $G=K A^+K_c$, where $A^+=\exp(\mathfrak{a}^+)$ (see \cite[\S 8.2.2]{BQ.book}). For $g \in G$, the middle factor in its $KA^+K_c$ factorisation is uniquely defined which allows us to define the Cartan projection 
$$\kappa: G \to \mathfrak{a}^+$$
which sends $g$ to the unique $x \in \mathfrak{a}^+$ such that $g \in K\exp(x)K_c$.

\subsubsection{Jordan projection} Every element $g$ of $G$ admits a Jordan decomposition as a commuting product $g=g_eg_hg_u$ where $g_e$ is elliptic, $g_h$ is semisimple and conjugated to an element of $\exp(\mathfrak{a}^+)$, and $g_u$ is unipotent. The element of $\exp(\mathfrak{a}^+)$ to which $g_h$ is conjugated is uniquely defined allowing us define the \emph{Jordan projection}:
$$\lambda: G \to \mathfrak{a}^+$$
which sends $g$ to the unique element $\lambda(g) \in \mathfrak{a}^+$ such that $g_h$ is conjugate to $\exp(\lambda(g))$.

\subsubsection{Representations}\label{subsub.rep}
For this paragraph we refer the reader to \cite[\S 6]{AMS}, \cite[\S 12]{borel-tits}, \cite{benoist.linear1}, and \cite[Ch. 8]{BQ.book}. Let $(V,\rho)$ be a finite dimensional linear representation of $G_c$. The weights of $(V,\rho)$ are the characters $\hat{\chi}: A \to \R_+^*$ such that the associated weight space $V_{\chi}=\{v\in V \colon \forall a \in A, \rho(a)v=\hat{\chi}(a)v\}$ is non-trivial. If $(V,\rho)$ is irreducible, then its set of weights admits a maximal element $\hat{\chi}_\rho$ (for the partial order given by $\chi_1 \leq \chi_2$  if any only if $\chi_2- \chi_1$ is a non-negative linear combination of positive roots) which is called the highest weight of $(V,\rho)$.  A representation $(V,\rho)$ is said to be \emph{proximal} if $\dim(V_{\hat{\chi}_{\rho}})=1$.

The following lemma singles out a collection of proximal representations of $G_c$. Thanks to these, the study of properties of elements of $G_c$ (such as Cartan/Jordan projections) boils down to the study of (simultaneous) linear algebraic properties of those representations.
\begin{lemma}(Tits \cite{tits.replin})\label{lemma.tits.distinguished}
For each $i=1,\ldots,d_S$, there 
exists a proximal irreducible real representation $(V_{i},\rho_{i})$ of $G_c$ with highest weight $\hat{\chi}_{i}$ such that $\chi_i$ is a multiple of the fundamental weight $\omega_{i}$. \end{lemma}
We note that for $i=1,\ldots,d_S$, (differentials of) all the other weights of $(V_{i},\rho_{i})$ consist of $(\chi_{i} -\alpha_{i})$ and others of the form $\chi_{i} -\alpha_{i}- \sum_{\beta \in \Pi}n_{\beta}\beta$ where $n_{\beta} \in \mathbb{N}$. In particular, for all $g \in G$ and $i=1,\ldots,d_S$, $\rho_{i}(g)$ is a proximal linear transformation of $V_{i}$ if and only if $\alpha_{i}(\lambda(g))>0$. Let $d=\dim \mathfrak{a}$.  We complete the representations given by the previous lemma by adding $d-d_S$ non-trivial representations $(V_i,\rho_i)$ of dimension one with weights $\chi_i$ for $i=d_S+1,\ldots, d$ so that $\{\chi_{i} : i=1,\ldots, d\}$ forms a basis of $\mathfrak{a}^\ast$. Then the mapping $a \to (\chi_{1}(a), \ldots,\chi_{d}(a))$ is an isomorphism of real vector spaces $\mathfrak{a} \to \mathbb{R}^{d}$. We will refer to the collection of representations $(V_i,\rho_i)$ for $i=1,\ldots,d$ as \textit{distinguished representations}.

For $i=1,\ldots,d$, we will consider the norms $\|.\|_{i}$ on $V_{i}$'s given by the next lemma originating from the work of Mostow \cite{mostow.self.adjoint}. We formulate a statement adapted to our purposes thanks to \cite{BQ.book}.

\begin{lemma}[Mostow norms]\label{lemma.mostow.norms}
Let $(V,\rho)$ be an irreducible real representation of $G_c$. Let $\hat{\chi}_{\rho,i}$ for $i=1\ldots,s=\dim V$ be the weights of $A$ in $V$. Then, there exists a Euclidean norm $\|.\|$ on $V$ such that for any $g \in G_c$, the logarithms of the singular values of $\rho(g)$ are given by $\chi_{\rho,i}(\kappa(g))$ for $i=1\ldots,s$. The same holds for moduli of eigenvalues of $\rho(g)$ replacing the Cartan projection $\kappa$ with the Jordan projection $\lambda$.
\end{lemma}

\begin{proof}
This follows from the existence of good norms \cite[\S 8.4.1]{BQ.book} together with \cite[Lemma 8.8]{BQ.book}.
\end{proof}

We borrow the following statement from Benoist--Quint \cite[Corollary 8.20]{BQ.book}.

\begin{corollary}[Uniform continuity of Cartan projection] \label{Cartan.stability}
Let $G$ be a real reductive group and $\kappa: G \to \mathfrak{a}^{+}$ be a Cartan projection of $G$. For every pair of nonempty compact subsets $L_1$ of $G$ and $L_2$ of $G_c$, there exists a compact subset $M$ of $\mathfrak{a}$ such that for every $g \in G$ we have $\kappa(L_1gL_2) \subseteq \kappa(g)+M$. 
\end{corollary}

\subsubsection{Loxodromic elements in $G$}

Let $G$ be a real reductive group and $\mathfrak{a}_S^{++}$ be the relative interior of the the intersection of the Weyl chamber $\mathfrak{a}^{+}$ with $\mathfrak{a}_S$.

\begin{definition} An element $g \in $G is said to be loxodromic if the orthogonal projection in $\mathfrak{a}$ of $\lambda(g)$ on $\mathfrak{a}_S$ belongs to $\mathfrak{a}_S^{++}.$ 
\end{definition}
By \S \ref{subsub.rep}, for $g \in G$ to be loxodromic is equivalent to asking that $\rho_i(g)$ is a proximal transformation in $\GL(V_{i})$ for each $i=1,\ldots,d_S$.
In the notation of \cite{BQ.book}, this notion coincides with the notion of $\Pi$-proximal element. 
As was done for linear maps in \S \ref{subsub.projective}, we quantify this notion as follows.  

\begin{definition} Let $r \geq \varepsilon>0$. An element $g \in $G is said to be $(r,\varepsilon)$-proximal in $G$ (or $(r,\varepsilon)$-loxodromic) if $\rho_i(g)$ is $(r,\varepsilon)$-proximal for each $i=1,\ldots,d_S$. 
\end{definition}

As observed by Benoist in \cite{benoist.linear1} (we borrow the formulations from \cite{breuillard-sert}), using the representation-theoretic bridge exposited in \S \ref{subsub.rep}, one gets the multidimensional counterparts to the above Lemma \ref{lemma.radius.vs.norm} and Proposition \ref{prop.products.proximals}.

\begin{lemma}\label{loxodromy.implies.Cartan.close.to.Jordan}
Let $G$ be a real reductive group and let $r>0$. Then there exists a constant  $C_{r}>0$ such that if $g\in G_c$ is $(r,\varepsilon)$-proximal in $G_c$ for some $\varepsilon \in (0,r]$, then $\|\lambda(g)-\kappa(g)\|\le C_r$. 
\end{lemma}



\begin{proposition}\label{Best}Let $G$ be a real reductive group. For every $r > 0 $, there exists a constant $C_r>0$ such that for every $\ell \in \N$, if $g_{1},\ldots, g_{\ell}$ are elements of $G_c$ that are $(r,\varepsilon)$-proximal in $G_c$ and have the property that $d(v_{\rho_{i}(g_{j})}^{+}, H_{\rho_{i}(g_{j+1})}^{<}) \geq 6r$ for all $j=0,\ldots, \ell-1$ (where we write $g_0:=g_\ell$), and for all $i=1,\ldots,d_S$, then for all $n_{1}, \ldots, n_{\ell} \in \N$, the product $g_{\ell}^{n_{\ell}}\ldots g_{1}^{n_{1}}$ is $(2r,2\varepsilon)$-proximal in $G_c$ and satisfies
\begin{equation*}
\left\|\lambda(g_{\ell}^{n_{\ell}}\ldots g_{1}^{n_{1}}) - \sum_{i=1}^{\ell} n_{i}\lambda(g_{i})\right\| \leq C_r \ell.
\end{equation*}
\end{proposition}

Analogously to Definition \ref{defnarepsSchottky1}, we have:

\begin{definition}[Schottky family in $G$]
Let $G$ be a real reductive group and let $r \geq \varepsilon >0$ be given. Then:
\begin{enumerate}[(i)]
\item
A subset $S$ of $G_c$ is said to be an $(r,\varepsilon)$-Schottky family in $G_c$, if for each $i=1,\ldots,d_S$ the set $\rho_{i}(S)$ is an $(r,\varepsilon)$-Schottky family.\\
\item
For $\eta \leq r$, a collection $S$ of $(r,\varepsilon)$-loxodromic elements is said to be \emph{$\eta$-narrow} if for every fixed $i=1,\ldots,d_S$ the attracting points $x_{i,g}^+ \in \mathbf{P}(V_i)$ are all within distance $\eta$ of one another as $g$ varies throughout $S$,  and if for every fixed $i=1,\ldots,d_S$ the repelling hyperplanes $H_{i,g}^< \subset \mathbf{P}(V_i)$ are within Hausdorff distance $\eta$ of one another as $g$ varies throughout $ S$.
\end{enumerate}
\end{definition}
When for each $i=1,\ldots,d_S$ a point $x_i \in  \mathbf{P}(V_i)$ and a hyperplane $H_i \subset \mathbf{P}(V_i)$ are understood, sometimes we will say that an $(r,\varepsilon)$-Schottky family is \emph{$\eta$-narrow around $((x_i),(H_i))$} if all attracting points $x_{i,g}^+$ are within $\eta/2$ distance of the corresponding point $x_i$ and if the analogous condition also holds for the repelling hyperplanes $H_{i,g}^<$ and their corresponding hyperplanes $H_i$. 

The first statement of the following corollary is the reductive group version of Corollary \ref{corol.family.to.semigroup.onerep}. The first clause follows easily from the latter result and, the statement concerning the narrowness of the semigroup, from Proposition \ref{prop.products.proximals}.  The second statement follows directly from the definitions.

\begin{corollary}\label{corol.family.to.semigroup}
Let $G$ be a real reductive group. Then the following properties hold:
\begin{enumerate}[(i)]
\item
Let $r>4\varepsilon>0$. Then the semigroup $\Gamma_S$ generated by an $(r,\varepsilon)$-Schottky family $S$ in $G$ is an $(r/2,2\varepsilon)$-Schottky family. If, furthermore, $S$ is $\eta$-narrow for some $\eta \geq 0$, then $\Gamma_S$ is $(\eta +2\epsilon)$-narrow.
\item
Let $E\subset G$ be an $\eta$-narrow collection of  $(r,\varepsilon)$-loxodromic elements, where $\varepsilon, \eta>0$ and $r>4\max\{ \varepsilon,\eta\}$. Then $E$ is an $(r/4,\varepsilon)$-Schottky family. \qed
\end{enumerate}
\end{corollary}

A direct corollary of Lemma \ref{loxodromy.implies.Cartan.close.to.Jordan} and Proposition \ref{Best} is the following additivity property of Cartan projections.

\begin{corollary}\label{corol.additive.cartan}
 Let $G$ be a real reductive group and $r >0$. Then, there exists a constant $C_r>0$ such that for every $\varepsilon \in (0,r]$ and any $(r,\varepsilon)$-Schottky family $S\subseteq G$, for all $\ell \in \N$, all integers $n_1,\ldots, n_\ell \geq 1$ and all $g_1,\ldots,g_\ell \in S$, we have
$$
\left\|\kappa(g_\ell^{n_\ell} \ldots g_1^{n_1})-\sum_{i=1}^\ell n_i \kappa(g_i)\right\| \leq C_r \ell. 
$$
\qed
\end{corollary}


%
%

\subsection{Ergodic theory and thermodynamic formalism.}

We end this section by briefly reviewing some relevant material from ergodic theory and subadditive thermodynamic formalism.

\subsubsection{Shift spaces and Lyapunov exponents.}

Given a finite nonempty set $\I$ we define $\Sigma_\I:=\I^\N$ and equip this set with the infinite product topology, with respect to which it is compact and metrisable. For every $x=(x_r)_{r=1}^\infty$ and $n \geq 1$ we let $x|_n$ denote the word $x_1x_2\cdots x_n \in \I^n$. Conversely, if a word $\iii \in \I^n$ is specified then $[\iii]\subset \Sigma_\I$ denotes the set of all $x \in \Sigma_\I$ such that $x|_n=\iii$. Given $x=(x_r)_{r=1}^\infty$ we define an element $\sigma x\in \Sigma_\I$ by $\sigma x:=(x_{r+1})_{r=1}^\infty$. This defines a continuous surjective transformation $\sigma \colon \Sigma_\I \to \Sigma_\I$. We let $\mathcal{M}_\sigma(\Sigma_\I)$ denote the set of all $\sigma$-invariant Borel probability measures on $\Sigma_\I$. If $(p_i)_{i \in \I}$ is any probability vector then the measure $\nu=(\sum_{i \in \I} \delta_i)^{\N}$ belongs to $\Sigma_\I$ and will be called the Bernoulli measure corresponding to the probability vector $(p_i)_{i \in \I}$. If $(p_i)_{i \in \I}$ is constant then its associated Bernoulli measure will be called the \emph{uniform Bernoulli measure} on $\Sigma_\I$. For each $\mu \in \mathcal{M}_\sigma(\Sigma_\I)$ we write $h(\mu)$ for the entropy of $\mu$ with respect to $\sigma$.

If $V$ is a $d$-dimensional real vector space equipped with an inner product then for every $(A_i)_{i \in \I} \in \GL(V)^\I$,  every $\mu \in \mathcal{M}_\sigma(\Sigma_\I)$ and every $k=1,\ldots,d$ we define the $k^{th}$ Lyapunov exponent of $(A_i)_{i \in \I}$ with respect to $\mu$ to be the quantity 
\begin{equation}\label{eq:def-lyap}\lambda_k\left((A_i)_{i \in \I}; \mu\right):=\lim_{n \to \infty} \frac{1}{n} \int_{\Sigma_\I} \log \sigma_k\left(A_{x|_n}\right) d\mu(x).\end{equation}
The existence of the above limit is guaranteed as follows. For every $k=1,\ldots,d$ and $B_1, B_2 \in \GL(V)$ the inequality $\prod_{\ell=1}^k \sigma_\ell(B_1B_2)\leq \prod_{\ell=1}^k \sigma_\ell(B_1) \cdot  \prod_{\ell=1}^k \sigma_\ell(B_2)$ follows from the identity $ \prod_{\ell=1}^k \sigma_\ell(B)\equiv \|\wedge^k B\|$, where $\|\wedge^kB\|$ denotes the norm of $B$ with respect to the inner product structure on $\wedge^k V$ induced by the inner product on $V$. This inequality guarantees the existence of the limit 
\begin{equation}\label{eq:l-limit}\lim_{n \to \infty} \frac{1}{n} \sum_{\ell=1}^k \int_{\Sigma_\I} \log \sigma_\ell\left(A_{x|_n}\right) d\mu(x)\end{equation}
 for every $k$ via subadditivity, and the existence of the limit \eqref{eq:def-lyap} follows by considering the difference between two instances of \eqref{eq:l-limit}. These observations also imply, via the subadditive ergodic theorem, that when the measure $\mu$ is ergodic with respect to $\sigma$ we have 
 \[\lim_{n \to \infty} \frac{1}{n}\log \sigma_k\left(A_{x|_n}\right)=\lambda_k((A_i)_{i \in \I} ; \mu)\]
 for $\mu$-a.e. $x \in \Sigma_\I$. When $(A_i)_{i \in \I}$ is understood we will often denote $\lambda_k((A_i)_{i \in \I}; \mu)$ more simply by $\lambda_k(\mu)$.
  
 The following theorem of Guivarc'h and Raugi (see \cite[Corollary 9]{GuRa}), which extends the earlier work of Furstenberg,  will be of particular use when studying the Lyapunov exponents of Bernoulli measures:
 \begin{theorem}\label{th:gu-ra}
 Let $V$ be a finite-dimensional real vector space, let $(A_i)_{i \in \I} \in \GL(V)^\I$ be proximal and strongly irreducible and let $\nu$ be a fully-supported Bernoulli measure on $\Sigma_\I$. Then $\lambda_1((A_i)_{i \in \I}; \nu)> \lambda_2((A_i)_{i \in \I}; \nu)$.
 \end{theorem}

In our considerations in Section \ref{sec.core}, we will also be considering random walks on the reductive groups themselves. There is a corresponding Lyapunov vector in the Weyl chamber (which in fact controls the Lyapunov exponents in all representations). More precisely, given a real reductive group $G$, a finite set $\mathcal{I}$ and an ordered sequence of elements $(g_i)_{i \in \mathcal{I}}$ in $G$, for any ergodic $\mu \in \mathcal{M}_{\sigma}(\Sigma_{\mathcal{I}})$, for $\mu$-a.e. $x \in \Sigma_\mathcal{I}$, the sequence of normalized Cartan projections $\frac{1}{n}\kappa(A_{x|_n})$ converges to a fixed element of $\mathfrak{a}^+$ which we denote by $\vec{\lambda}
(\mu)$ and call the \textit{Lyapunov vector} of $\mu$ in $\mathfrak{a}^+$. Thanks to Lemma \ref{lemma.tits.distinguished} (and the subsequent discussion), the convergence follows again from subadditive ergodic theorem applied in distinguished representations.

\subsubsection{Pressure, affinity dimension and Lyapunov dimension.}

We now review and extend the notions of affinity dimension and Lyapunov dimension which were defined in the introduction and connect them more directly with ergodic theory.

Let $(T_i)_{i \in \I}$ be an affine IFS acting on $\R^d$ which is contracting with respect to some fixed norm on $\R^d$, and write each $T_i$ in the form $T_ix=A_ix+v_i$ where $A_i \in \GL_d(\R)$ and $v_i \in \R^d$. As in the introduction, for every $s \geq 0$ and $B \in \GL_d(\R)$ we define 
\[\varphi^s(B):=\left\{\begin{array}{cl}\sigma_1(B) \cdots \sigma_{\lfloor s\rfloor}(B) \sigma_{\lceil s\rceil}(B)^{s-\lfloor s\rfloor}&\text{if }0 \leq s \leq d,\\ |\det B|^{\frac{s}{d}}&\text{if }d \geq s,\end{array}\right.\] 
The inequality $\varphi^s(B_1B_2) \leq \varphi^s(B_1)\varphi^s(B_2)$ is well-established and may be found in \cite{Fa88}. For every $\mu \in \Sigma_\I$ and $s \geq 0$ we define
\[\Lambda_s((A_i)_{i\in \I}; \mu):=\lim_{n \to \infty}\frac{1}{n}\int_{\Sigma_\I} \log \varphi^s(A_{x|_n}) d\mu(x)\]
where we note that the existence of the limit is again guaranteed by subadditivity. It follows directly from the relevant definitions that 
\[\Lambda_s((A_i)_{i\in \I}; \mu)=\left\{\begin{array}{cl}\sum_{\ell=1}^{\lfloor s\rfloor} \lambda_\ell((A_i)_{i \in \I} ; \mu) + (s-\lfloor s\rfloor) \lambda_{\lceil s\rceil} ((A_i)_{i \in \I};\mu)
&\text{if }0 \leq s \leq d,\\ \frac{s}{d}\sum_{\ell=1}^d \lambda_\ell ((A_i)_{i \in \I} ; \mu)&\text{if }d \geq s.\end{array}\right.\] 

For every $s \geq 0$ we define the pressure of $(T_i)_{i \in \I}$ with parameter $s$ to be the quantity
\[P((A_i)_{i \in \I}; s):=\lim_{n \to \infty} \frac{1}{n} \log \sum_{\iii \in \I^n} \varphi^s(A_\iii)\]
which is well-defined by subadditivity. The following result adapts earlier work of A. K\"aenm\"aki in \cite{Ka04}; a direct proof is briefly outlined in \cite[\S3.1]{MoSe19}.
\begin{proposition}\label{pr:aff-fund}
Let $(T_i)_{i \in \I}$ be an affine IFS acting on $\R^d$ which is contracting with respect to some fixed norm on $\R^d$, and let $(A_i)_{i \in \I} \in \GL_d(\R)^\I$ denote the vector of linearisations of $(T_i)_{i \in \I}$.  Then:
\begin{enumerate}[(i)]
\item
The function $[0,\infty) \to \R$ defined by $s \mapsto P((A_i)_{i \in \I} ; s)$ is continuous and strictly decreasing, is positive at $s=0$, and tends to $-\infty$ in the limit as $s \to \infty$.
\item
For each fixed $\mu \in \mathcal{M}_\sigma(\Sigma_\I)$ the function $[0,\infty) \to \R$ defined by $s \mapsto h(\mu) + \Lambda_s((A_i)_{i\in \I}; \mu)$ is continuous and strictly decreasing, is non-negative at $s=0$, and  tends to $-\infty$ in the limit as $s \to \infty$.
\end{enumerate}
\end{proposition}
The preceding proposition guarantees the soundness of the following definition:
\begin{definition}
Let $(T_i)_{i \in \I}$ be an affine IFS acting on $\R^d$ which is contracting with respect to some fixed norm on $\R^d$ and let $(A_i)_{i \in \I} \in \GL_d(\R)^\I$ denote the vector of linearisations of $(T_i)_{i \in \I}$. Then the \emph{affinity dimension} of $(T_i)_{i \in \I}$, denoted $\dimaff (T_i)_{i\in \I}$, is defined to be the unique $s \geq 0$ such that $P((A_i)_{i \in \I}; s)=0$. For every $\mu \in \mathcal{M}_\sigma(\Sigma_\I)$ the \emph{Lyapunov dimension} of $(T_i)_{i \in \I}$ with respect to $\mu$, denoted $\dimlyap  ((T_i)_{i \in \I}; \mu)$, is likewise defined to be the unique $s \geq 0$ such that $h(\mu) + \Lambda_s((A_i)_{i \in \I}; \mu)=0$. 
\end{definition} 
The following result is a special case of a very general variational principle proved by Y.-L. Cao, D.-J. Feng and W. Huang in \cite{CaFeHu08}:
\begin{proposition}\label{pr:yadda}
Let $(T_i)_{i \in \I}$ be an affine IFS acting on $\R^d$ which is contracting with respect to some fixed norm on $\R^d$ and let $(A_i)_{i \in \I} \in \GL_d(\R)^\I$ denote the vector of linearisations of $(T_i)_{i \in \I}$. Then for every $s \geq 0$,
\begin{equation}\label{eq:eqm}P((A_i)_{i \in \I} ; s) = \sup\left\{h(\mu)+\Lambda_s((A_i)_{i \in \I}; \mu) \colon \mu \in \mathcal{M}_\sigma(\Sigma_\I)\right\}\end{equation}
and this supremum is attained by a measure which is ergodic with respect to $\sigma$.
\end{proposition}
We refer to a measure which attains the supremum \eqref{eq:eqm} as a \emph{$\varphi^s$-equilibrium state} of $(A_i)_{i \in \I}$. We finally note the following variational principle, originally due to A. K\"aenm\"aki in \cite{Ka04}, which may be obtained directly from the combination of the two propositions above:
\begin{corollary}
Let $(T_i)_{i \in \I}$ be an affine IFS acting on $\R^d$ which is contracting with respect to some fixed norm on $\R^d$. Then
\[\dimaff (T_i)_{i \in \I} = \sup\left\{ \dimlyap ((T_i)_{i \in \I}; \mu) \colon \mu \in \mathcal{M}_\sigma(\Sigma_\I)\right\}\]
and this supremum is attained by a measure which is ergodic with respect to $\sigma$.
\end{corollary}

\section{Approximation by Schottky subsystems}\label{sec.core}

\subsection{The core technical result}

We now introduce the key technical result underlying Theorem \ref{th:main}. Before the statement, we introduce a further notation: for an integer $k \in \N$ and a set of finite words $W$, we denote by $W^k$ all $k$-fold concatenations of elements of $W$, i.e.~ $W^k=\{\iii_1 \ldots \iii_k : \iii_j \in W\}$. Similarly, for a finite word $\iii$ we denote by $\iii^k$ the concatenation $\iii \iii \cdots \iii$ of $k$ copies of $\iii$. We also say that a word $\iii \in \I^*$ is \emph{suffixed} by a word $\kkk \in \I^*$ if either $\iii=\kkk$ or if $\iii=\jjj\kkk$ for some $\jjj \in \I^*$.

\begin{theorem}\label{thm.key.intrinsic}
Let $G$ be a real reductive group, $\mathcal{I}$ a finite set and $(g_i)_{i \in \mathcal{I}}$ an ordered sequence of elements of $G$ such that $\{g_i \colon i \in \mathcal{I}\}$ generates a Zariski-dense subsemigroup of $G$. Suppose that there exist $\alpha ,\beta>0$ and $x \in \mathfrak{a}^+$ with the following properties: for infinitely many natural numbers $n$ there exists a set of words $\mathcal{I}_n \subseteq \mathcal{I}^n$ such that $\# \mathcal{I}_n \geq e^{\alpha n}$ and such that for every $\jjj \in \mathcal{I}_n$ we have
\begin{equation}\label{eq.cartan.bound0}
\left\|\frac{1}{n}\kappa(g_\jjj)-x\right\| \leq \beta.    
\end{equation}
Then there exists $r>0$ with the following property: for every fixed word $\iii_0 \in \mathcal{I}^\ast$ and any positive real numbers $\varepsilon, \alpha', \alpha'', \beta'$ satisfying $\varepsilon<r$, $\alpha'<\alpha<\alpha''$ and $\beta'>\beta$, there exist an integer $n \geq 1$ and a set $\mathcal{J} \subseteq \mathcal{I}^n$ such that:
\begin{enumerate}[(i)]
\item\label{it:tech-zd}
For every integer $\ell \geq 1$ the semigroup generated by $\{g_\jjj \colon \jjj \in \mathcal{J}^\ell\}$ is Zariski-dense in $G_c$.
\item\label{it:tech-schottky}
The collection $\{g_\jjj \colon \jjj \in \mathcal{J}^\ast\}$ is an $\varepsilon$-narrow $(r,\varepsilon)$-Schottky family.
\item\label{it:tech-cartan}
For every $\ell \geq 1$ and $\jjj \in \mathcal{J}^\ell$, $  \|\frac{1}{n \ell}\kappa(g_\jjj)-x\| \leq \beta'$.
\item\label{it:tech-cardinality} The cardinality of $\mathcal{J}$ satisfies $e^{\alpha' n} \leq \# \mathcal{J}\leq e^{\alpha'' n}$.
\item\label{it:tech-suffix}
The set $\mathcal{J}$ consists entirely of words suffixed by $\iii_0$.
\end{enumerate}
\end{theorem}

\begin{remark}
Trivially the conclusions \eqref{it:tech-cardinality} and \eqref{it:tech-suffix} also hold for $\mathcal{J}^\ell$, replacing $n$ with $n\ell$.
\end{remark}

\begin{remark}
The main technical ingredients used in proving the above result are, among others, due to Tits \cite{tits.free}, Goldsheid--Margulis \cite{goldsheid-margulis}, and Abels--Margulis--Soifer \cite{AMS}. The techniques used in combining these results go back to the work of Benoist in \cite{benoist.linear1,benoist.linear2} and of Quint in \cite{Quint.divergence}, who applied these ideas to study asymptotic properties of subsemigroups of linear groups.\end{remark}
\begin{remark}
Similar statements to \eqref{it:tech-schottky}--\eqref{it:tech-cardinality} above appear in \cite{sert.LDP} expressed in probabilistic language, where they are used to establish a large deviation principle (LDP) for random matrix products. These versions of \eqref{it:tech-schottky}--\eqref{it:tech-cardinality} were later obtained by Park \cite{park} in the more general setting of subshifts of finite type and H\"{o}lder cocycles (see also Mohammadpour \cite{mohammadpour} who later noted, in a similar way as in \cite{sert.LDP}, the consequences of these properties in application to multifractal formalism). 
\end{remark}

We will ultimately apply Theorem \ref{thm.key.intrinsic} in the following form which is adapted to the specific task of proving Theorem \ref{th:main}:
\begin{corollary}\label{corol.key.with.mu}
Let $G$ be a real reductive group, $\mathcal{I}$ a finite set and $(g_i)_{i \in \mathcal{I}}$ an ordered sequence of elements in $G$ such that $\{g_i \colon i \in \mathcal{I}\}$ generates a Zariski-dense semigroup in $G$. Let $\mu \in \mathcal{M}_\sigma(\Sigma_{\mathcal{I}})$ be an ergodic measure and $\iii_0$ a fixed word in $\mathcal{I}^\ast$. Then, there exist $r>0$ such that for every $\varepsilon \in (0,r)$ and $N \geq 1$, we can find an integer $n \geq N$, a set $\mathcal{J} \subseteq \mathcal{I}^n$ satisfying

\begin{enumerate}[(i)]
\item\label{it:co-zd}  For every integer $\ell \geq 1$ semigroup generated by $\{g_\jjj \colon \jjj \in \mathcal{J}^\ell\}$ is Zariski-dense in $G_c$.
\item\label{it:co-schottky} the collection $\{g_\jjj \colon \jjj \in \mathcal{J}^\ast\}$ is an $\varepsilon$-narrow $(r,\varepsilon)$-Schottky family.
\item\label{it:co-cartan} For every $\ell \geq 1$ and $\jjj \in \mathcal{J}^\ell$, $  \|\frac{1}{n \ell}\kappa(g_\jjj)-\vec{\lambda}(\mu)\| \leq \varepsilon$.
\item\label{it:co-cardinality} The cardinality of $\mathcal{J}$ satisfies $|\frac{1}{n} \log \# \mathcal{J} - h(\mu)|\leq \varepsilon$.
\item\label{it:co-suffix} The set $\mathcal{J}$ consists of words suffixed by $\iii_0$.
\end{enumerate} 
\end{corollary}

Before commencing the proof of Theorem \ref{thm.key.intrinsic} we establish a few technical results which will be employed in addition to the results and notions exposited in \S \ref{sec.preliminaries}.

\subsection{Further preliminary results}

The next lemma (due to Benoist in \cite{benoist.linear1}) will allow us to find generic elements that will be used to ensure Zariski-density in Theorem \ref{thm.key.intrinsic} in an $(r,\varepsilon)$-Schottky family.

\begin{lemma}[Zariski density of narrow Schottky families]\label{lemma.narrow.dense}
Let $\Gamma$ be a Zariski-dense subsemigroup of a real reductive group $G$ and let $g \in \Gamma$ be a loxodromic element. Then, there exists $r>0$ depending only on $\Gamma$ such that for every $\eta>0$ and $\varepsilon>0$, there exists an $(r,\varepsilon)$-Schottky semigroup in $\Gamma$ which is $\eta$-narrow around $((x_{i,g}^+),H_{i,g}^<)_{i=1,\ldots,d_S}$ and which is Zariski-dense in $\Gamma$.
\end{lemma}
\begin{proof}
This follows readily by combining \cite[Lemma 3.6(v)]{benoist.linear1} and Corollary \ref{corol.family.to.semigroup}.
\end{proof}

The next observation (which will be familiar to specialists) will allow us to pass to certain convenient sub-semigroups while retaining the same Zariski closure.
\begin{lemma}\label{lemma.dense.powers}
Let $\Gamma_0$ be a finitely generated semigroup in $\GL_d(\R)$ with Zariski-connected Zariski-closure $G$. Then for any Zariski-dense subsemigroup $\Gamma<\Gamma_0$ there exist $\gamma_1,\ldots,\gamma_t \in \Gamma$ such that for every $t$-tuple of positive integers $(n_i)_{i=1,\ldots,t}$ the semigroup generated by $\gamma_1^{n_1},\ldots, \gamma_t^{n_t}$ is Zariski-dense in $G$.
\end{lemma}

To prove this lemma we first require the following observation which will also be employed in proving the main result.

\begin{lemma}\label{lemma.reduce.to.fg}
Let $\Gamma_0<G$ be as in the previous lemma and $S$ a Zariski-dense subset of $\Gamma_0$. Then there exists a finite subset of $S$ which generates a Zariski-dense subsemigroup of $\Gamma_0$. 
\end{lemma}

In particular, in this lemma, if $S$ is a subsemigroup of $\Gamma_0$, then it contains a finitely generated Zariski-dense subsemigroup. 

\begin{proof}
Let $H\leq G$ be a Zariski-connected linear algebraic subgroup which arises as the identity component of the Zariski closure of a subgroup of $G$ generated by finitely many elements of $S$, and which is maximal among all such subgroups. It is clear that a maximal subgroup of this type exists and is unique. If the conclusion of the lemma does not hold then $H$ is a proper subgroup of $G$ such that any finite collection of elements of $S$ generates a semigroup whose (Zariski) identity component is contained in $H$. It follows that $S$ normalises $H$ and, by Zariski-density, $H$ 
is normal in $G$. Furthermore, by \cite[Lemma 4.2]{tits.free} applied to the group $\Gamma_0^{\pm}$ generated by $\Gamma_0$ (which is also finitely generated) there exists $n_0 \geq 1$ such that for every $\gamma \in \Gamma^{\pm}_0$, the Zariski closure of the group generated by $\gamma^{n_0}$ is Zariski-connected. Then for each $\gamma \in S$,
the image $\overline{\gamma}^{n_0}$ of $\gamma^{n_0}$ in $G/H$ generates a finite and Zariski-connected group and therefore it is trivial, i.e.~$\overline{\gamma}^{n_0}=\iiid_{G/H}$. Since $S$ is Zariski-dense in $G$, the same relation holds for every $g \in \Gamma_0^\pm$, i.e.~ $\overline{g}^{n_0}=\iiid_{G/H}$. Since $\Gamma_0^\pm$ is finitely generated, by the Burnside--Schur theorem this implies that the image of $\Gamma_0^\pm$ in $G/H$ is finite. By Zariski-density we deduce that $H$ has finite index in $G$ and by Zariski-connectedness, we have $G=H$, a contradiction.
\end{proof}

\begin{proof}[Proof of Lemma \ref{lemma.dense.powers}]
By Lemma \ref{lemma.reduce.to.fg}, up to replacing $\Gamma$ by a finitely generated Zariski-dense subsemigroup of itself, we can suppose that $\Gamma$ is finitely generated.
By \cite[Lemma 4.2]{tits.free} applied to the group $\Gamma^{\pm}$ generated by $\Gamma$ (which is also clearly finitely generated), there exists $n_\Gamma\geq 1$ such that for every $\gamma \in \Gamma^{\pm}$, the Zariski closure of the group generated by $\gamma^{n_\Gamma}$ is Zariski-connected. Note, on the other hand, that for every $n \geq 1$, the semigroup generated by elements $\{\gamma^n \colon \gamma \in \Gamma\}$ is still Zariski-dense in $\Gamma^{\pm}$. Indeed, the latter Zariski closure -- which is an algebraic group, which we denote by $R$ -- is equal to the Zariski-closure of the group generated by $\{\gamma^n \colon \gamma \in \Gamma^\pm\}$. Since the set $\{\gamma^n \colon \gamma \in \Gamma^\pm\}$ is invariant by conjugation by $\Gamma^{\pm}$, the Zariski-closure of the group which it generates  is invariant by the Zariski-closure of $\Gamma^{\pm}$, which is $G$. Therefore, $R$ is a normal linear algebraic subgroup of $G$ such that the image $\overline{\Gamma}$ of $\Gamma$ in $G/R$ is Zariski-dense, torsion and finitely generated. It then  follows from  the Burnside--Schur theorem that $\overline{\Gamma}$ is finite and hence since $G$ is Zariski-connected, we have $G=R$ as required.
In view of this, by another application of Lemma \ref{lemma.reduce.to.fg}, we can choose $\gamma_1',\ldots,\gamma_t' \in \Gamma$ such that that the semigroup generated by $\gamma_1=(\gamma_1')^{n_\Gamma},\ldots, \gamma_t=(\gamma_t')^{n_\Gamma}$ is Zariski-dense in $G$. Now suppose that $m_1,\ldots,m_t$ are positive integers. By the choice of $n_\Gamma$, for each $i$ the Zariski-closure of the semigroup generated by $\gamma_i^{m_i}$ is connected, and clearly contains the semigroup generated by $\gamma_i$ as a Zariski-connected, finite-index subgroup. It follows that for each $i$ these two Zariski closures are identical, and in particular we have $\gamma_i \in \overline{\langle \gamma_i^{m_i}\rangle}^Z$ for every $i=1,\ldots,t$. The Zariski-closure of the semigroup generated by $\gamma_1^{m_1},\ldots,\gamma_t^{m_t}$ thus contains the Zariski-closure of the semigroup generated by $\gamma_1,\ldots,\gamma_t$, which is $G$, and we conclude that the semigroup generated by $\gamma_1^{m_1},\ldots,\gamma_t^{m_t}$  is Zariski-dense in $G$ as required.
\end{proof}

\subsubsection{Abels-Margulis-Soifer}


An important and non-trivial fact about Zariski-dense semigroups of reductive Lie groups is that every such semigroup contains a loxodromic element, see \cite{benoist-labourie,goldsheid-margulis,Prasad}. The following result of Abels-Margulis-Soifer \cite[Thm. 6.8]{AMS} extends this observation by showing that loxodromic elements are abundant in a certain sense: loxodromic elements can be obtained from an arbitrary semigroup element $g$ by perturbing $g$ with an element coming from a fixed finite set. This result will be essential in controlling the cardinality of the set $\mathcal{J}$ constructed in Theorem \ref{thm.key.intrinsic}.  

\begin{theorem}[Abels-Margulis-Soifer \cite{AMS}]\label{thm.AMS} 
Let $G$ be a Zariski-connected real reductive group and $\Gamma$ a Zariski-dense subsemigroup. Then there exists $r>0$ depending only on $\Gamma$ such that for every  $\varepsilon \in (0, r]$ we may choose a finite subset $F=F(r,\varepsilon,\Gamma)$ of $\Gamma$ with the property that for every $\gamma \in G$, there exists $f \in F $ such that $f\gamma$ is $(r,\varepsilon)$-proximal in $G$.
\end{theorem}

Before proceeding with the proof of Theorem \ref{thm.key.intrinsic}, we lastly introduce the following notation: for a subset $W$ of $\I^*$ we denote by $\vec{W}$ the set of right products of elements of $W$, i.e.~ $\vec{W}=\{g_{i_1}\ldots g_{i_n} \colon i_1\cdots i_n \in W\}$.

\subsection{Proof of Theorem \ref{thm.key.intrinsic}}
Define $\Gamma_1$ to be the semigroup generated by $\{g_i \colon i \in \mathcal{I}\}$. We begin the proof by fixing certain quantities which depend on $\Gamma_1$. For each $\gamma \in \Gamma_1$ we let $n_\gamma \in \N$ denote the minimal length of a word $\iii=i_1\cdots i_m \in \Sigma_\mathcal{I}^\ast$ such that $\gamma=g_{i_1}\ldots g_{i_m}$.

\textbullet ${}$ \textit{Groundwork: establishing parameters and constants.}  Define $\Gamma_c:=\Gamma_1 \cap G_c$. Clearly, $\Gamma_c$ is Zariski-dense in $G_c$. Let $N_1$ denote the number of connected components of $G$. Since $\Gamma_1$ is Zariski-dense in $G$ we may choose  $N_1$ elements $h_1,\ldots,h_{N_1} \in \Gamma_1$ such that every connected component of $G$ contains one of the elements $h_i$. Let $r_0>0$ be the minimum of the two constants $r>0$ given respectively by Lemma \ref{lemma.narrow.dense} and Theorem \ref{thm.AMS} applied to $\Gamma_c$. Let $d_S \geq 0$ be the semisimple rank of $G$, and let $(\rho_1,V_1),\ldots, (\rho_{d_S}, V_{d_S})$ be the distinguished representations of $G_c$ given by the application of Lemma \ref{lemma.tits.distinguished}. 

Now fix a positive real number $\varepsilon<r_0/64$, and fix some further quantities which depend on $\Gamma_1$ and on $\varepsilon$ as follows. Let $F_1=F_1(r,\varepsilon)$ be the finite subset of $\Gamma_c$ given by the application of Theorem \ref{thm.AMS} to the semigroup $\Gamma_c$ and group $G_c$ and let $N_2$ denote the cardinality of $F_1$. Let $N_3$ be minimal cardinality of a finite $\varepsilon/2$-net of the compact metric space $\prod_{i=1}^{d_S} P(V_i) \times P(V_i^\ast)$. Let $\iii_0$ be the arbitrary word specified in the statement of the proposition and let $g_{\iii_0} \in \Gamma_1$ be the corresponding group element. By replacing $\iii_0$ with a suitable power of itself if necessary, we assume without loss of generality that $g_{\iii_0} \in \Gamma_c$. Define
\[M:=|\iii_0|+\max\{n_\gamma \colon \gamma \in \{h_1,\ldots,h_{N_1}\}\} + \max\{n_f \colon f \in F_1\}.\]
By applying Corollary \ref{Cartan.stability} with $L_1=\{\id, h_1,\ldots,h_{N_1}\} \cup F_1$ and $L_2=\{g_{\iii_0}, \id\}$, it follows that there exists a real number $K>0$ such that each of the quantities
\[\|\kappa(g) - \kappa(h_ig)\|,\qquad \|\kappa(g)-\kappa (fg)\|, \qquad \|\kappa(g)-\kappa (gg_{\iii_0})\| \]
is bounded above by $K$ for every $g \in G$, every $i=1,\ldots,N_1$ and every $f \in F_1$.

We now choose a large integer $n_0 \in \mathbb{N}$ and proceed in several stages to modify the set $\mathcal{I}_{n_0}$ so as to construct the required integer $n \in \mathbb{N}$ and the collection  $\mathcal{J} \subseteq \mathcal{I}^{n}$ with properties \eqref{it:tech-zd}--\eqref{it:tech-suffix}.\\

\textbullet ${}$ \textit{First modification: reduction to a connected component.} Since $G$ has exactly $N_1$ connected components, by the pigeonhole principle we may choose a subset $W_1(n_0)'$ of $\mathcal{I}_{n_0}$ with cardinality at least $|\mathcal{I}_{n_0}|/N_1$ such that $\vec{W}_1'(n_0)$ lies in a single connected component of $G$. Clearly we may write this connected component as $h_j^{-1}G_c$ for some $j \in \{1,\ldots,N_1\}$. Choose a word $\lll=\ell_1\cdots \ell_n \in \I^*$, with length not greater than $M$, such that $g_{\ell_1}\cdots g_{\ell_n}=h_j$ and define $W_1(n_0):=\{\lll \jjj \colon \jjj \in W_1'(n_0)\}$. We see that $W_1(n_0)$ consists of words of length exactly $n_0+|\lll|$, has cardinality at least $|\mathcal{I}_{n_0}|/N_1$, and satisfies $\vec{W}_1(n_0)\subset G_c$. In view of the hypothesis \eqref{eq.cartan.bound0} and Corollary \ref{Cartan.stability} we have
\[\left\|\kappa(g)-n_0 x\right\|\leq n_0\beta+K\] 
for every $g \in \vec{W}_1(n_0)$. \\

\textbullet ${}$ \textit{Second modification: adjoining the fixed word $\iii_0$.}
Let $g_{\iii_0} \in \Gamma_c$ be the group element corresponding to the arbitrary word $\iii_0$ specified in the statement of the theorem. The second modification simply consists of adjoining $\iii_0$ as a suffix to the words in $W_1(n_0)$. Define $W_2(n_0):=\{\kkk \iii_0 \colon \kkk \in W_1(n_0)\}$.  Clearly $W_2(n_0)$ consists of words of length precisely $n_0+|\lll|+|\iii_0|$ and has the same cardinality as $W_1(n_0)$. The set $\vec{W}_2(n_0)$ is clearly a subset of $G_c$, and by a further application of Corollary \ref{Cartan.stability} we have
\[\left\|\kappa(g)-n_0 x\right\|\leq n_0\beta+2K\] 
for every $g \in \vec{W_2}(n_0)$.\\

\textbullet ${}$ \textit{Third modification: obtaining loxodromy using the theorem of Abels, Margulis and Soifer.} By Theorem \ref{thm.AMS}, for each element $g$ of $\vec{W}_2(n_0)$ there exists $f \in F_1$ such that the product $fg$ is $(r_0,\varepsilon)$-loxodromic. By a second pigeonhole argument, it follows that there exist an element $f_0 \in F_1$ and a subset $W_3(n_0)'$ of $W_2(n_0)$ with cardinality at least $|W_2(n_0)|/|F_1|$ such that $f_0g$ is $(r_0,\varepsilon)$-loxodromic for every $g \in \vec{W}_3'(n_0)$. Let $\lll'$ be a word of length $n_{f_0}$ which represents the group element $f_0$, and define $W_3(n_0):=\{\lll'\kkk \colon \kkk \in W_3'(n_0)\}$. Clearly $W_3(n_0)$ consists only of words whose length is precisely $n_0+|\lll|+|\iii_0|+n_{f_0}$ and has cardinality at least $|W_2(n_0)|/N_2=|W_1(n_0)|/N_2 \geq |\mathcal{I}_{n_0}|/N_1N_2$, and every element of $g \in \vec{W}_3(n_0)$ is $(r_0,\varepsilon)$-loxodromic, belongs to $G_c$ and satisfies
\begin{equation}\label{eq.cartan.bound2}\left\|\kappa(g)-n_0 x\right\|\leq n_0\beta+3K.  
\end{equation}
\smallskip

\textbullet ${}$ \textit{Fourth modification: narrowing.} 
As noted earlier, $d_S \geq 0$ is the semisimple rank of $G$ and $(\rho_1,V_1),\ldots,(\rho_{d_S}, V_{d_S})$ are the distinguished representations of $G_c$ given by Lemma \ref{lemma.tits.distinguished}. Using compactness we may partition $\prod_{i=1}^{d_S} P(V_i) \times P(V_i^\ast)$ into $N_3$ sets each of radius not greater than $\varepsilon/2$. By a third pigeonhole argument there exist a subset $W_4(n_0)$ of $W_3(n_0)$ with cardinality at least $|W_3(n_0)|/N_3\geq |\mathcal{I}_{n_0}|/N_1N_2N_3\geq e^{n\alpha}/N_1N_2N_3$ and an element $\mathcal{P}$ of the partition such that for every $g \in \vec{W}_4(n_0)$ the tuple of attracting points $x_{i,g}^+$ and repelling hyperplanes $H_{i,g}^<$ in the representations $\rho_i$ for $i=1,\ldots,d_S$ is an element of $\mathcal{P}$. By removing some elements from $W_4(n_0)$ if necessary, we choose $W_4(n_0)$ in such a manner that its cardinality satisfies the bounds
\begin{equation}\label{eq.w.cardinality}e^{\alpha n_0}/N_1N_2N_3\leq |W_4(n_0)| \leq 2e^{\alpha n_0}.\end{equation}
We observe in particular that $\vec{W}_4(n_0)$ is an $\varepsilon$-narrow set of $(r_0,\varepsilon)$-loxodromic elements. Therefore, by Corollary \ref{corol.family.to.semigroup} and the properties of $W_3(n_0)$ established previously, $W_4(n_0)$ consists of words whose length is precisely $n_1:=n_0+|\lll|+|\iii_0|+n_{f_0}\leq n_0+M$, and $\vec{W}_4(n_0)$ is an $\varepsilon$-narrow $(r_0/4,\varepsilon)$-Schottky family in $G_c$. The property \eqref{eq.cartan.bound2} clearly persists for elements of $\vec{W}_4(n_0)\subseteq \vec{W}_3(n_0)$, so in particular
\begin{equation}\label{eq.cartan.bound3}\left\|\kappa(g)-n_1 x\right\|\leq (n_1-n_0)\|x\|+n_0\beta+3K\leq n_1\beta + M\|x\|+3K\end{equation}
for every $g \in \vec{W}_4(n_0)$.

At this stage of the argument we have successfully constructed a set of words which could be shown to satisfy \eqref{it:pr-cardinality}--\eqref{it:pr-suffix}, but which \emph{a priori} need not satisfy \eqref{it:pr-zd}, since we have no knowledge of the Zariski closure of the semigroup generated by $\vec{W}_4(n_0)$ other than that it is contained in $G_c$. In the next stage of the argument we will start to address \eqref{it:pr-zd}, but in doing so we will sacrifice some control over the lengths of the words and over their Cartan vectors, creating defects which will be remedied in further stages of the construction.  \\

\textbullet ${}$ \textit{Fifth modification: obtaining Zariski density via Tits' lemmas.} 
Let $\Gamma_1^\pm$ denote the group generated by $\Gamma_1$, and note that this group is clearly also generated by the finite set $\{g_i \colon i \in \I\}$. The group $\Gamma_1^{\pm} \cap G_c$ is Zariski-dense in $G_c$, contains the semigroup $\Gamma_c$ and is finitely generated since it is a finite-index subgroup of the finitely-generated group $\Gamma_1^{\pm}$. By \cite[Lemma 4.2]{tits.free} there exists $n_\Gamma \in \N$ such that for every $\gamma \in \Gamma_1^{\pm} \cap G_c$, the Zariski-closure of the group generated by $\gamma^{n_\Gamma}$ is Zariski-connected.
Fix an element $g' \in \vec{W}_4(n_0)$ arbitrarily and define $g:=(g')^{n_\Gamma}$. By Lemma \ref{lemma.narrow.dense}, there exists a Zariski-dense $(r_0,\varepsilon)$-Schottky subsemigroup $\Gamma_g$ in $\Gamma_c$ that is $\varepsilon$-narrow around $((x_{i,g}^+),(H_{i,g}^<))$. By Lemma \ref{lemma.reduce.to.fg}, $\Gamma_g$ contains a finitely-generated Zariski-dense subsemigroup. It then follows by the same argument as was used in the proof of Lemma \ref{lemma.dense.powers} that the subsemigroup $\Gamma_{g, n_\Gamma}$ of $\Gamma_g$ which is generated by the $n_\Gamma^{th}$ powers of the elements of $\Gamma_g$ is Zariski-dense in $G_c$. By \cite[Proposition 4.4]{tits.free} (see also \cite[Lemme 4.3]{benoist.linear1}), there exists a Zariski-closed proper subset $\mathcal{F}_g$ of the semisimple quotient $G_c/Z(G_c)$ such that the union of all Zariski-closed, Zariski-connected proper subgroups of $G_c/Z(G_c)$ which contain the projection of $g$ is itself contained in $\mathcal{F}_g$. In particular, the pre-image $\tilde{\mathcal{F}}_g^c$ in $G_c$ of the complement $\mathcal{F}_g^c$ is nonempty and Zariski open. Since $G_c$ is Zariski-connected and hence irreducible, the intersection  $\Gamma_{g,n_\Gamma} \cap \tilde{\mathcal{F}}_g^c$  is non-empty and Zariski-dense in $G_c$. By Lemma \ref{lemma.reduce.to.fg} we may choose finitely many elements $\theta_1,\ldots, \theta_p$ of $ \Gamma_{g,n_\Gamma}$ so that the semigroup generated by $\vec{W}_4(n_0)\cup\{\theta_1,\ldots,\theta_p\}$ is Zariski-dense in $G_c$. Indeed, for any $i=1,\ldots,p$, the semigroup generated by $\vec{W}_4(n_0)\cup \{\theta_i\}$ is Zariski-dense in $G_c/Z(G_c)$ by the defining property of $\mathcal{F}_g$. Consequently, choosing the elements $\theta_1,\ldots,\theta_p$ so that the projection to $G_c/[G_c,G_c]$ of the semigroup which they generate  is Zariski-dense yields a list of elements with the desired properties. Now we choose words $\iii_{\theta_1},\ldots,\iii_{\theta_p}$ representing the elements $\theta_1,\ldots,\theta_p$ and set $W_5(n_0)$ to be the union of $W_4(n_0)$ and $\{\iii_{\theta_i} : i=1, \ldots, p\}$. Since the elements $\theta_i$ are elements of $\Gamma_g$, which is an $(r_0,\varepsilon)$-Schottky semigroup that is $\varepsilon$-narrow around $((x_{i,g}^+),(H_{i,g}^<))$, each $\theta_i$ is $(r_0,\varepsilon)$-loxodromic. The set $\vec{W}_4(n_0) \cup \{\theta_i : i=1,\ldots, p\}$ therefore consists of $(r_0/4,\varepsilon)$-loxodromic elements and is $2\varepsilon$-narrow. Since $r_0/32>\varepsilon>0$, it follows from the second clause of Corollary \ref{corol.family.to.semigroup} that $\vec{W}_5(n_0)$ is a $2\varepsilon$-narrow $(r_0/16, \varepsilon)$-Schottky family that generates a Zariski-dense subgroup of $G_c$. 

At this stage $W_5(n_0)$ comes closer to satisfying the required properties insofar as the Zariski closure of the semigroup generated by the corresponding elements is properly controlled. However, for the newly added elements $\iii_{\theta_i}$ the word-length and Cartan projection  are \emph{a priori} uncontrolled, and these words also \emph{a priori} lack the required suffix $\iii_0$. These problems will be remedied in two further modifications of the set $W_5(n_0)$.\\

\textbullet ${}$ \textit{Sixth modification: equalising the word lengths.}
We are in a position to apply Lemma \ref{lemma.dense.powers} to the semigroup  generated by $\vec{W}_5(n_0)$, which we denote by $\Gamma_5$. By the first clause of Corollary \ref{corol.family.to.semigroup} and by the inequalities $r_0/64>\varepsilon>0$, the semigroup $\Gamma_5$ is a $4\varepsilon$-narrow $(r_0/32,2\varepsilon)$-Schottky family. Denote by $\gamma_1',\ldots,\gamma_t'$  the elements of $\Gamma_5$ given by Lemma \ref{lemma.dense.powers} and set $\gamma_i:=(\gamma_i')^{n_\Gamma}$ for $i=1,\ldots, t$. Clearly the elements $\gamma_1,\ldots,\gamma_t$ also have the property stipulated in the conclusion of the lemma. For $i=1,\ldots, t$ choose an integer $m_i$ such that we may choose a word $\iii_{\gamma_i} \in \mathcal{I}^{m_i}$ representing the element $\gamma_i$. Denote by $m$ the product $m_1 \ldots m_t n_\Gamma n_1$, where recall that $n_0+|\lll|+|\iii_0|+n_{f_0}$ is the length of every word in $W_4(n_0)$. We now set $W_6(m)$ to be the set $\{\iii_{\gamma_i}^{m/m_i} : i=1,\ldots,t\} \cup W_4(n_0)^{m/n_1}$ and let $\Gamma_6$ denote the semigroup generated by $\vec{W}_6(m)$. Being a subset of $\Gamma_5$, $\Gamma_6$ is a $4\varepsilon$-narrow $(r_0/32,2\varepsilon)$-Schottky family.\\


\textbullet ${}$ \textit{Seventh modification: controlling the Cartan vectors.}
Fix an element $g \in \vec{W}_4(n_0)^{m/n_1}$ which has the form $g=h^{n_\Gamma}$ for some $h \in \vec{W}_4(n_0)^{m/n_\Gamma n_1}$,  and let $\iii_g \in W_4(n_0)^{m/n_1}\subseteq  \mathcal{I}^{m}$ be a representative word which is suffixed by the distinguished word $\iii_0$. By the defining property of $n_\Gamma$, the Zariski-closure $\overline{\langle g \rangle}^Z$ of the group $\langle g \rangle$ generated by $g$ is Zariski-connected. Since $g$ is an ($m/n_1$)-fold product of elements of $\vec{W}_4(n_0)$, and since $\vec{W}_4(n_0)$ is an $(r_0/4,\varepsilon)$-Schottky family, it follows from \eqref{eq.cartan.bound3} together with Corollary \ref{corol.additive.cartan} that we have
\begin{align}\nonumber
    \left\|\kappa(g)-mx\right\|& \leq m\beta + \left(\frac{m}{n_1} \left(M\|x\| + 3K + C_{r_0/4} \right)\right)\\\label{eq.cartan.control2}
    & \leq m\beta + \left(\frac{m}{n_0} \left(M\|x\| + 3K + C_{r_0/4} \right)\right)
\end{align}
where the constant $C_{r_0/4}$ is provided by Corollary \ref{corol.additive.cartan} and depends only on the group $G$ and the parameters $r_0$ and $\varepsilon$. We will now show that for a large enough integer $k$, the set $W_7(k)$ defined by
$$
W_7(k):=W_4(n)^{km/n_1} \cup \bigcup_{j=1}^t \left\{ \iii_{\gamma_j}^{m/m_j}\iii_g^{(k-1)}\right\}\footnote{In an earlier version, we had imposed the suffix $i_0$ at step 5 in an erroneous manner, we thank Jialun Li who spotted this and mentioned to us that it can easily be imposed at this stage.},
$$
satisfies the conditions as required.\\

\textbullet ${}$ \textit{End: Verifying \eqref{it:tech-zd}--\eqref{it:tech-suffix}:}  We begin by observing that every element of $W_7(k)$ is a word of length precisely $mk$, which is to say that $W_7(k)\subseteq \mathcal{I}^{mk}$.

Let $\Gamma_7$ denote the semigroup generated by $\vec{W}_7(k)$.  To see \eqref{it:tech-zd} we first observe that since every element of $\vec{W}_7(k)$ is an element of $\Gamma_c$, we have $\overline{\Gamma}_7^Z \subseteq G_c$. On the other hand $g^{k} \in \vec{W}_7(k)$ and since $\overline{\langle g \rangle}^Z$ is Zariski-connected, $g$ belongs to the Zariski-closure of $\Gamma_7$. 
For each $j=1,\ldots,t$ the product $\gamma_j^{m/m_j}g^{k-1}$ is an element of $\vec{W}_7(mk)$ and hence of $\overline{\Gamma_7}^Z$. Since $\overline{\Gamma_7}^Z$ is a group it follows that $\gamma_j^{m/m_j} \in \overline{\Gamma_7}^Z$. In particular $\overline{\Gamma}_7^Z$ contains the Zariski closure of the semigroup generated by $\gamma_1^{m/m_1},\ldots,\gamma_t^{m/m_t}$, which is $G_c$. We thus have $G_c \subseteq \overline{\Gamma_7}^Z \subseteq G_c$ as required to prove \eqref{it:tech-zd}. Indeed, an easy modification of this argument demonstrates that $\vec{W}_7(mk)^\ell$ generates a Zariski-dense subsemigroup of $G_c$ for every $\ell \geq 1$. To conclude that property \eqref{it:tech-schottky} also holds we simply observe that $\Gamma_7$ is a subset of $\Gamma_5$ which is a $4\varepsilon$-narrow $(r_0/32,\varepsilon)$-Schottky family, and hence $\Gamma_7$  is a $4\varepsilon$-narrow $(r_0/32,\varepsilon)$-Schottky family also. Clearly this suffices to establish \eqref{it:tech-schottky}.

To obtain \eqref{it:tech-cartan} we argue as follows. Since $\Gamma_6$ is a $(r/32, \varepsilon)$-Schottky family, for each $j=1,\ldots,t$ we have
\begin{align*}\MoveEqLeft[3]{\left\|\kappa \left(\gamma_j^{m/m_j} g^{k-1}\right)-mkx\right\|}&\\
&\leq 2C_{r_0/32}+ \|\gamma_j^{m/m_j}-mx\| + (k-1)\|\kappa(g)-mx\|\\
&< 2C_{r_0/32}+ \|\gamma_j^{m/m_j}-mx\| + km\beta + \left(\frac{km}{n_0}\left(M\|x\|+3K+C_{r_0/4}\right)\right)\end{align*}
where we have used Corollary \ref{corol.additive.cartan} and  \eqref{eq.cartan.bound3}. If $\theta$ is any element of $W_4(n)^{km/n_1}$ then in a similar manner
\[\left\|\kappa(\theta) - mkx\right\|
\leq km\beta + \left(\frac{km}{n_0} \left(M\|x\| + 3K + C_{r_0/4} +C_{r_0/32} \right)\right).
\]
It follows that if $k$ is chosen large enough that
\[\frac{1}{mk} \left(2C_{r_0/32} + \max_{1 \leq j \leq t} \left\|\gamma_j^{m/m_j}-mx\right\| \right) <\frac{\beta'-\beta}{2}\]
and also $n_0$ was chosen large enough that
\[\frac{1}{n_0} \left(M\|x\|+3K+C_{r_0/4}+C_{r_0/32}\right)<\frac{\beta'-\beta}{2} \]
then \eqref{it:tech-cartan} is satisfied. Since we are free to choose $k\geq 1$ arbitrarily, and since the desired choice of $n_0$ depends only on quantities which were available at the beginning of the argument when $n_0$ was chosen, these constraints can be met. We have established \eqref{it:tech-cartan}.

To establish \eqref{it:tech-cardinality} we note the inequalities
\[|W_7(mk)| \leq \left|W_4(n_0)\right|^{mk/n_1} + t\leq 2^{mk/n_1} e^{mkn_0\alpha/n_1} +t  \leq 2^{mk/n_0} e^{mk\alpha}+t\]
and
\[|W_7(mk)| \geq \left|W_4(n_0)\right|^{mk/n_1} \geq \left(\frac{|\mathcal{I}_{n_0}|}{N_1N_2N_3}\right)^{mk/n_1} \geq\frac{e^{n_0mk\alpha/n_1}}{(N_1N_2N_3)^{mk/n_0}}\]
arising from the construction of $W_7(k)$ together with the inequalities \eqref{eq.w.cardinality}. Now, by construction
\[ \frac{n_0}{n_1}\geq 1-\frac{M}{n_1}\geq 1-\frac{M}{n_0},\]
and it is compatible with the constraints which arose when treating \eqref{it:tech-cartan} to require that
$n_0$ is chosen large enough to satisfy
\[e^{n_0\alpha/n_1}\geq e^{\alpha(1- M/n_0)} \geq (N_1N_2N_3)^{1/n_0} e^{\alpha'}\]
and
\[2^{1/n_0}<e^{\frac{\alpha''-\alpha}{2}}\]
since these constraints too depend only on quantities which were available at the time when $n_0$ was chosen. 
Clearly we may also suppose that $k$ is chosen large enough to satisfy
\[ e^{\frac{mk(\alpha+\alpha'')}{2}}+t\leq e^{mk\alpha''}.\]
These inequalities combine to provide the required upper and lower bounds on $|W_7(k)|$.  Finally, it is clear that \eqref{it:tech-suffix} is satisfied since both the word $\iii_g$ considered in the definition of $W_7(k)$, and all elements of $W_4(n_0)$, are by construction suffixed with $\iii_0$. The proof is complete. \qed

\subsection{Proof of Corollary \ref{corol.key.with.mu}}
The result will follow from Theorem \ref{thm.key.intrinsic} once we show that for every $\varepsilon>0$, we can find a sequence of sets $\mathcal{I}_n$ of words of length $n$ with cardinality bounded from below by $e^{n(h(\mu)-\varepsilon)}$ and such that every $\iii \in \mathcal{I}_n$ satisfies $\|\frac{1}{n}\kappa(g_{\iii})-\vec{\lambda}(\mu)\|<\varepsilon$.

These will follow readily from the subadditive ergodic theorem and from the Shannon--Mcmillan--Breiman theorem as follows. By the subadditive ergodic theorem we have for $\mu$-a.e. $x \in \Sigma_{\mathcal{I}}$
$$
\lim_{n \to \infty} \frac{1}{n} \kappa(g_{x|_n}) \to  \vec{\lambda}(\mu)
$$
and by Shannon--McMillan--Breiman theorem, we have for $\mu$-a.e. $x \in \Sigma_{\mathcal{I}}$
$$
\lim_{n \to \infty} \frac{1}{n}\log \mu([x|_n]) \to -h(\mu).
$$
It follows by convergence in measure that there exists $N_0=N_0(\varepsilon)$, which may be chosen arbitrarily large, such that for every $n \geq N_0$ we may choose a measurable subset $\Omega_n \subseteq \Sigma_{\mathcal{I}}$ satisfying $\mu(\Omega_n)>1-\varepsilon$ and such that for every $x \in \Omega_n$ 
\begin{equation}\label{eq.entropy.bound}
e^{-n(h(\mu)+\varepsilon/2)} \leq  \mu\left([x|_n]\right) \leq e^{-n(h(\mu)-\varepsilon/2)}, 
\end{equation}
and
\begin{equation}\label{eq.cartan.bound}
\left\|\frac{1}{n}\kappa(g_{x|_n})-\vec{\lambda}(\mu)\right\|\leq\varepsilon.
\end{equation}
We now define $\mathcal{I}_n:=\{x|_n \colon x \in \Omega_n\}$ for every $n \geq N_0$. It follows directly from \eqref{eq.entropy.bound} that $\# \mathcal{I}_n \geq (1-\varepsilon)e^{n(h(\mu)-\varepsilon/2)} \geq e^{n(h(\mu)-\varepsilon)}$ for all large enough $n$. Together with \eqref{eq.cartan.bound} this completes the proof. \qed

%
%

\section{Proof of Theorem \ref{th:main}}\label{sec.pf.main.thm}


\subsection{Preparation}\label{subsec.extrinsic}

Before beginning the proof of Theorem \ref{th:main} we translate the results of the previous section into an extrinsic form in Proposition \ref{pr:new-prop} below.
The deduction is by standard use of representation theory of real reductive groups as exposited in \S \ref{subsec.reductive} (mainly Mostow's Lemma \ref{lemma.mostow.norms}). 

\begin{proposition}\label{pr:new-prop}
Let $\I$ be a finite set, let $(A_i)_{i \in \mathcal{I}} \in \GL_d(\R)^\I$ be completely reducible, let $\mu$ be an ergodic shift-invariant measure on $\Sigma_\I$, and let $\iii_0 \in \I^*$. Let $G$ denote the Zariski closure of the semigroup $\{A_\iii \colon \iii \in \I^*\}$.
Then for every $\varepsilon>0$ there exist $n \geq 1$ and a set $\J \subset \I^n$ such that the following properties hold:
\begin{enumerate}[(i)]
\item\label{it:pr-zd}
The Zariski closure of the semigroup $\{A_\jjj \colon \jjj \in \J^*\}$
is precisely $G_c$.
\item\label{it:pr-cardinality}
The cardinality of $\J$ is at least $e^{n(h(\mu)-\varepsilon)}$.
\item\label{it:pr-cartan}
For every $\ell \in \{1,\ldots,d\}$, we have
\[\left|\sum_{r=1}^\ell\log \sigma_r(A_\jjj) - n|\jjj|\sum_{r=1}^\ell\lambda_r((A_i)_{i \in \mathcal{I}},\mu)\right| \leq n|\jjj|\varepsilon\]
for every $\jjj\in \J^*$.
\item\label{it:pr-dom}
For every $k \in \{1,\ldots,d-1\}$ such that $\lambda_k((A_i)_{i \in \mathcal{I}},\mu)>\lambda_{k+1}((A_i)_{i \in \mathcal{I}},\mu)$,  $(A_\jjj)_{\jjj \in \J}$ is $k$-dominated.
\item\label{it:pr-suffix}
Every $\jjj \in \J$ is suffixed by $\iii_0$.
\end{enumerate}
\end{proposition}

\begin{remark}
In \eqref{it:pr-cartan} the length $|\jjj|$ of the word $\jjj$ should be understood as its length as a word over the alphabet $\J$ and not as its length as a word over the alphabet $\I$. (In the latter sense, the word $\jjj$ is $n$ times longer).
\end{remark}
\begin{remark}\label{rk.folding}
Regarding \eqref{it:pr-dom} above, it also follows from Corollary \ref{corol.key.with.mu} that the proximality index of the semigroup $\{\wedge^k A_\jjj \colon \jjj \in \J^*\}$ is equal to the proximality index of $\wedge^k G$. Moreover, if $k \in \{1,\ldots,d-1\}$ is such that $\lambda_k((A_i)_{i \in \mathcal{I}},\mu)>\lambda_{k+1}((A_i)_{i \in \mathcal{I}},\mu)$, then the proximality index of $\wedge^k G$ (and hence of $\{\wedge^k A_\jjj \colon \jjj \in \J^*\}$) is one.

Accordingly, one might hope to obtain the following stronger form of  \eqref{it:pr-dom}: if the proximality index of $\wedge^k G$ is one, then $\{A_\jjj \colon \jjj \in \J^*\}$ is $k$-dominated (or equivalently, $\{\wedge^k A_\jjj \colon \jjj \in \J^*\}$ is 1-dominated). However this outcome cannot be guaranteed in general since $\wedge^k G$ may have proximality index equal to $1$ while the two Lyapunov exponents $\lambda_k((A_i)_{i \in \mathcal{I}},\mu)$ and $\lambda_{k+1}((A_i)_{i \in \mathcal{I}},\mu)$ coincide. This can arise as a consequence of a ``folding'' phenomenon wherein the Weyl chamber of $\mathfrak{a}^+$ of $G$ is embedded piecewise affinely in that of $\GL_d(\R)$ and where $\vec{\lambda}(\mu)$ may happen to belong to a folding hyperplane in $\mathfrak{a}^+$. (For a further discussion of folding, see e.g.~ \cite[\S 3.10 \& Fig.~ 1]{breuillard-sert} where the same phenomenon appears as the reason why convexity of joint spectrum in $\GL_d(\R)$, defined in that article, may fail to hold if one only assumes irreducibility; see also the article \cite{morris-sert.etds}, where the folding phenomenon plays the role in the existence of multiple equilibrium states for the singular value potential.)

The following example illustrates this phenomenon. Consider the representation $\rho_1:\SL_2(\R) \times \SL_2(\R) \to \SL_4(\R)$ given by left and right multiplication on $\Mat_{2 \times 2}(\mathbb{R}) \simeq \R^4$, and let $\rho_2: \SL_4(\R) \to \SL_6(\R)$ be the exterior power representation on $\R^6 \simeq \wedge^2\R^4$. Then $\rho_2 \circ \rho_1: \SL_2(\R) \times \SL_2(\R) \to \SL_6(\R)$ has proximality index one and is a direct sum of the three-dimensional irreducible representations of the two $\SL_2(\R)$ factors. For $i=1,2$ let $\pi_i: \SL_2(\R) \times \SL_2(\R) \to \SL_2(\R)$ denote the representation given respectively by projection onto the first or second co-ordinate of the Cartesian product. If a shift-invariant probability measure $\mu$ is chosen on $(\SL_2(\R) \times \SL_2(\R))^\mathbb{N}$ such that the top Lyapunov exponents of the measures $\pi_1 {}_{\ast}\mu$ and $\pi_2 {}_{\ast}\mu$ are equal, then we will have $\lambda_1((\rho_2 \circ \rho_1)_\ast \mu)=\lambda_2((\rho_2 \circ \rho_1)_\ast \mu)$. Such a measure $\mu$ exists (\cite[Theorem 1.11]{breuillard-sert}) and can be chosen to be Bernoulli with finite support, such that the support generates a Zariski-dense subsemigroup of $\SL_2(\R) \times \SL_2(\R)$ (\cite[Th\'{e}or\`{e}me 3.5]{thi.thesis}). 

\end{remark}


\begin{proof}
Let $G$ be the Zariski-closure of the semigroup generated by $\{A_i : i \in \mathcal{I}\}$ and $G_c$ the real points of its connected component. Let $h(\mu)$ and $\vec{\lambda}(\mu)$ denote respectively the metric entropy of $\mu$ with respect to the shift $\sigma$ and the Lyapunov vector of $\mu$ belonging to a chosen Weyl chamber $\mathfrak{a}^+$ in a Cartan subspace $\mathfrak{a}$ in the Lie algebra $\mathfrak{g}$ of $G$. By complete reducibility, the group $G$ is a real reductive group and we are in a position to apply Corollary \ref{corol.key.with.mu}. We find that there exists a constant $r>0$ such that for every $\varepsilon \in (0,r)$ and $\iii_0 \in \mathcal{I}^\ast$, there exists $n \in \N$ and a subset $\mathcal{J} \subseteq \mathcal{I}^n$ satisfying the conclusions \eqref{it:co-zd}--\eqref{it:co-suffix} of Corollary \ref{corol.key.with.mu}. The conclusions \eqref{it:pr-zd}, \eqref{it:pr-cardinality} and \eqref{it:pr-suffix} follow directly, and \eqref{it:pr-dom} is an immediate consequence of \eqref{it:pr-cartan} as long as $\varepsilon$ is chosen sufficiently small (specifically, as long as $2\varepsilon$ is strictly smaller than every nonzero difference between pairs of successive Lyapunov exponents). We thus need only to prove \eqref{it:pr-cartan}. This will follow from \eqref{it:co-cartan} of Corollary \ref{corol.key.with.mu} thanks to the representation theoretic Lemma \ref{lemma.mostow.norms}. To establish  \eqref{it:pr-cartan} we will derive the estimate
\begin{equation}\label{eq:simple-iii}\left|\sum_{r=1}^\ell\log \sigma_r(A_\jjj) - n|\jjj|\sum_{r=1}^\ell\lambda_r((A_i)_{i \in \mathcal{I}},\mu)\right| \leq n|\jjj|\cdot \frac{\varepsilon}{2} + C_0\ell\end{equation}
where the constant $C_0$ is independent of $\ell$, $n$ and $\varepsilon$, and as long as $n$ was chosen sufficiently large this implies the result in the form claimed in the statement.

Let $(\chi_i)_{i=1,\ldots,d}$ denote the weights of the given representation, let us denote it by $\rho$, of $G_c$ on $\R^d$. By Lemma \ref{lemma.mostow.norms} applied to each irreducible subrepresentation of $\rho$, the collection of Lyapunov exponents (with multiplicities) of $\mu$ is given by the collection of evaluations,  with multiplicities, of $\vec{\lambda}(\mu)$ on all weights of $\rho$. In other words, we have the following equality of tuples, up to relabeling of indices,
\begin{equation}\label{eq.list.lyapunov}
    (\lambda_i((A_i)_{i \in \mathcal{I}},\mu))_{i=1,\ldots, d}=(\chi_i(\vec{\lambda}(\mu)))_{i=1,\ldots, d}.
\end{equation}

By the same Lemma \ref{lemma.mostow.norms}, for every $\jjj \in \mathcal{J}^\ast$, all logarithms of singular values of $g_\jjj$ are given by the values obtained by  the collection $(\chi_i)_{i=1,\ldots,d}$ of linear forms in $\mathfrak{a}^\ast$ evaluated at $\kappa(g_\jjj)$, up to a bounded additive constant $C_0$ (depending only on the representation $\rho$ of $G_c$) -- this appears due to possible change of the Euclidean structure in Lemma \ref{lemma.mostow.norms}.
Written differently, we have the following equality of tuples, up to a relabeling of indices,
\begin{equation}\label{eq.singulars}
    \left(\log \sigma_i( g_\jjj)\right)_{i=1,\ldots, d}=(\chi'_{i}(\kappa(g_\jjj)))_{i=1,\ldots,d}.
\end{equation}
where $\chi'_{i}(\kappa(g_\jjj))$'s are reals satisfying $|\chi_{i}(\kappa(g_\jjj))-\chi'_{i}(\kappa(g_\jjj))| \leq C_0$ for every $i=1,\ldots,d$.
The inequality \eqref{eq:simple-iii} now follows by combining \eqref{eq.list.lyapunov} and \eqref{eq.singulars} with \eqref{it:co-cartan} of Corollary \ref{corol.key.with.mu}, where the corollary is applied with $\varepsilon/2$ in place of $\varepsilon$. 
\end{proof}
The following lemma, which will be used to treat the final clause of Theorem \ref{th:main}, appeared previously in \cite{MoSh19} in a slightly different form. For completeness we include a proof.
\begin{lemma}\label{le:pablo}
Let $(T_i)_{i \in \I}$ be a finite collection of affine transformations of $\R^d$ all of which are contracting with respect to some fixed norm. Suppose that the strong open set condition holds for $(T_i)_{i \in \I}$. Then there exists a finite word $\iii_0 \in \Sigma_\I$ such that for every $m \geq 1$ and every nonempty set $\J\subset \I^m$, the iterated function system $(T_{\jjj \iii_0})_{\jjj \in \J}$ satisfies the strong separation condition.
\end{lemma}
\begin{proof}
Let $\threebar{\cdot}$ be a norm on $\R^d$ such that every $T_i$ is contracting with respect to $\threebar{\cdot}$ and let $X \subset \R^d$ denote the attractor of $(T_i)_{i \in \I}$. Throughout this proof we equip $\R^d$ with the metric induced by $\threebar{\cdot}$. Choose $\tau \in (0,1)$ such that $\threebar{T_ix-T_iy} \leq \tau\threebar{x-y}$ for all $x,y \in \R^d$ and $i \in \I$. By the strong open set condition there exists a nonempty open set $U \subset \R^d$ such that the sets $T_iU$ for $i \in \I$ are pairwise disjoint subsets of $U$ and such that additionally $X \cap U$ is nonempty. It is straightforward to see that these properties are retained if $U$ is replaced by the intersection of $U$ with a sufficiently large open ball centred at the origin, so without loss of generality we may assume additionally that $U$ is bounded. Now choose $x_0 \in X \cap U$ and $\kappa>0$ such that the open ball around $x_0$ with radius $\kappa$ is a subset of $U$, and choose $n\geq 1$ such that $\tau^n \diam U<\kappa$. By iteration of the relation $\bigcup_{i\in \I}T_iX=X$ we have $\bigcup_{\iii \in \I^n} T_\iii X=X$, so in particular there exists $\iii_0 \in \I^n$ such that $x_0 \in T_{\iii_0} X$. We now note that $x_0 \in T_{\iii_0}\overline{U}$ and that $\diam T_{\iii_0}\overline{U}<\kappa$, so in particular $T_{\iii_0}\overline{U}\subset U$.

Now let $m \geq 1$ and $\J\subset \I^m$ be arbitrary and let $Z$ denote the attractor of $(T_{\jjj \iii_0})_{\jjj \in \J}$. If $\jjj_1, \jjj_2 \in \J$ are distinct words, write 
$\jjj_1=\kkk j_1\lll_1$ and $\jjj_2= \kkk j_2\lll_2$ where $j_1, j_2 \in \I$ are distinct and where each of $\kkk$, $\lll_1$ and $\lll_2$ is either an element of $\I^*$ or the empty word. We then have 
\begin{align*}T_{\jjj_1 \iii_0}Z \cap T_{\jjj_2 \iii_0}Z &= T_\kkk \left(T_{j_1} T_{\lll_1} T_{\iii_0}Z \cap T_{j_2} T_{\lll_2} T_{\iii_0}Z\right)\\
&\subseteq  T_\kkk \left(T_{j_1} T_{\lll_1} T_{\iii_0}\overline{U} \cap T_{j_2} T_{\lll_2} T_{\iii_0}\overline{U}\right)\\
&\subseteq  T_\kkk \left(T_{j_1} T_{\lll_1}U \cap T_{j_2} T_{\lll_2}U\right)\subseteq T_\kkk(T_{j_1}U \cap T_{j_2}U)=\emptyset\end{align*}
so that $T_{\jjj_1 \iii_0}Z$ and $T_{\jjj_2 \iii_0}Z$ are disjoint, and this proves the strong separation condition.\end{proof}

\subsection{Proof of Theorem \ref{th:main}}
Fix $\delta>0$ throughout the proof, and choose a positive, non-integer real number $s$ such that
\[ \dimaff (T_i)_{i\in\I} -\delta< s < \dimaff (T_i)_{i \in \I}.\]
It follows from Proposition \ref{pr:aff-fund} that $P((A_i)_{i \in \I}, s)$ is strictly positive, so we may choose $\varepsilon>0$ small enough that 
\[P((A_i)_{i \in \I}, s)-4\varepsilon>0.\]
Let $\mu \in \mathcal{M}_\sigma(\Sigma_\I)$ be an ergodic $\varphi^{s}$-equilibrium state for $(A_i)_{i \in \I}$ and let $\mathsf{k}_0$ denote the set of all integers $k$ in the range $1,\ldots,d-1$ such that
\[\lambda_k((A_i)_{i\in \I}; \mu) > \lambda_{k+1}((A_i)_{i\in \I}; \mu)\]
or such that equivalently
\[\lambda_1\left(\left(\wedge^k A_i\right)_{i\in \I}; \mu\right) > \lambda_2\left(\left(\wedge^k A_i\right)_{i\in \I}; \mu\right).\]
If $(T_i)_{i \in \I}$ satisfies the strong open set condition let $\iii_0\in \I^*$ be the word provided by Lemma \ref{le:pablo}; otherwise, let $\iii_0 \in \I^*$ be fixed but arbitrary.

By Proposition \ref{pr:new-prop} applied to $(A_i)_{i \in \I}$, $\iii_0$, $\mu$ and $G$ there exist $m_0 \geq 1$ and $\J_0 \subset \I^{m_0}$ such that: the Zariski closure of $\{A_\kkk \colon \kkk \in \J_0^*\}$ is precisely $G_c$; the uniform Bernoulli measure $\nu_0$ on $\Sigma_{\J_0}$ has entropy at least $m_0(h(\mu)-\varepsilon)$; for every $k=1,\ldots,d$ we have
\begin{equation}\label{eq:1-dom-first}\left|\frac{1}{|\jjj|}\sum_{\ell=1}^{k}\log \sigma_\ell(A_\jjj) - m_0\sum_{\ell=1}^{k}\lambda_\ell((A_i)_{i\in \I};\mu)\right| \leq m_0\varepsilon\end{equation}
for all $\jjj \in \J_0^*$; for every $k \in \mathsf{k}_0$,  $(\wedge^k A_\jjj)_{\jjj \in \J_0}$ is $1$-dominated; and every element of $\J_0$ has $\iii_0$ as a suffix. In particular, by \eqref{eq:1-dom-first} it follows directly that
\begin{equation}\label{eq:lyap-gap-first}
\sum_{\ell=1}^{k}\lambda_\ell((A_\jjj)_{\jjj\in \J_0};\nu_0)\geq m_0\left(\sum_{\ell=1}^{k}\lambda_\ell((A_i)_{i\in \I};\mu)-\varepsilon\right)
\end{equation}
for all $k=1,\ldots,d$. This already suffices to establish most of the conclusions of Theorem \ref{th:main}, but at the moment we have no explicit knowledge of exactly which integers $k$ have the property that $(A_\jjj)_{\jjj \in \J_0}$ is $k$-dominated. To remedy this we shall apply Proposition \ref{pr:new-prop} a second time, using instead the measure $\nu_0$.

Let $\mathsf{k}$ denote the set of all $k \in \{1,\ldots,d-1\}$ such that 
\[\lambda_k((A_\jjj)_{\jjj\in \J_0}; \nu_0) > \lambda_{k+1}((A_\jjj)_{\jjj\in \J_0}; \nu_0).\]
By Theorem \ref{th:gu-ra}, if $(\wedge^k A_\jjj)_{\jjj \in \J_0}$ is proximal and strongly irreducible for some integer $k \in \{1,\ldots,d-1\}$ then
\[\lambda_1\left((\wedge^k A_\jjj)_{\jjj\in \J_0}; \nu_0\right) > \lambda_2\left((\wedge^k A_\jjj)_{\jjj\in \J_0}; \nu_0\right)\]
for that integer $k$, which in particular implies that $k \in \mathsf{k}$. Let $\mathsf{k}_1 \subset \{1,\ldots,d-1\}$ denote the set of all $k$ such that $(\wedge^k A_i)_{i \in \I}$ is strongly irreducible and proximal; by Theorem \ref{th:gm} (and the fact that $G_c$ is finite-index in $G$), this is precisely the set of all $k$ such that the exterior power representation $\wedge^k \colon G_c \to \GL(\wedge^k \R^d)$ is irreducible and proximal, and this in turn is precisely the set of all $k$ such that $(\wedge^k A_\jjj)_{\jjj \in \J_0}$ is strongly irreducible and proximal.
It follows from the preceding remarks that $\mathsf{k}_1\subset \mathsf{k}$. 

Now apply Proposition \ref{pr:new-prop} to $(A_\jjj)_{\jjj \in \J_0}$, an arbitrary word $\jjj_0 \in \J$, the uniform Bernoulli measure $\nu_0 \in \mathcal{M}_\sigma(\Sigma_{\J_0})$ and the group $G_c$. This yields an integer $m_1 \geq 1$ and a set $\mathcal{J} \subset\J_0^{m_1}\simeq \I^{m_0m_1}$ with the following properties: the Zariski closure of $\{A_\kkk \colon \kkk \in \J^*\}$ is precisely $G_c$; the uniform Bernoulli measure $\nu$ on $\Sigma_{\J}$ has entropy at least $m_1(h(\nu_0)-\varepsilon)$; for every $k =1,\ldots,d$ the inequality
\begin{equation}\label{eq:lyap-gap-second}
\sum_{\ell=1}^{k}\lambda_\ell((A_\jjj)_{\jjj\in \J};\nu)\geq m_1\left(\sum_{\ell=1}^{k}\lambda_\ell((A_\jjj)_{\jjj\in \J_0};\nu_0)-\varepsilon\right)
\end{equation}
is satisfied; for every $k \in \mathsf{k}$ the tuple $(\wedge^k A_\jjj)_{\jjj \in \J}$ is $1$-dominated; and every $\jjj \in \J$ has $\jjj_0$ as a suffix. Trivially every $\jjj \in \J$ has $\iii_0$ as a suffix (when identified with an element of $\I^*$) since $\jjj_0 \in \J_0$. 

We now claim that $\J$ satisfies properties \eqref{it:uniform}--\eqref{it:strongsep}. It is immediate from the preceding statements that \eqref{it:zar} holds, and also \eqref{it:gaps} follows since $\mathsf{k}_1\subseteq \mathsf{k}$. If $(T_i)_{i \in \I}$ does not satisfy the strong open set condition then \eqref{it:strongsep} is vacuously true, and otherwise $(T_\jjj)_{\jjj \in \J}$ satisfies the strong separation condition by Lemma \ref{le:pablo} and the fact that every word $\jjj \in \J\simeq \I^{m_1m_0}$ is suffixed by $\iii_0$ when considered as an element of $\iii_0$. It remains only to establish \eqref{it:uniform}. By Proposition \ref{pr:aff-fund}  it suffices to show that $P((A_\jjj)_{\jjj \in \J}, s)>0$, and by Proposition \ref{pr:yadda} this will follow if we show that
\[h(\nu) + \sum_{\ell=1}^{\lfloor s\rfloor} \lambda_\ell ((A_\jjj)_{\jjj \in \J_0}; \nu) + (s-\lfloor s\rfloor) \lambda_{\lceil s\rceil} ((A_\jjj)_{\jjj \in \J_0}; \nu)>0.\]
But this is now straightforward. The entropy $h(\nu)$ may be estimated by combining directly the inequalities arising from the two applications of Proposition \ref{pr:new-prop}:
\begin{equation}\label{eq:entryism}h(\nu) \geq m_1(h(\nu_0)-\varepsilon) \geq m_1(m_0(h(\mu)-\varepsilon)-\varepsilon) \geq m_0m_1(h(\mu)-2\varepsilon).\end{equation}
Similarly by the combination of \eqref{eq:lyap-gap-first} and \eqref{eq:lyap-gap-second}, 
\begin{align*}
\nonumber\sum_{\ell=1}^{\lfloor s\rfloor}\lambda_\ell((A_\jjj)_{\jjj \in \J}; \nu) &\geq m_1 \left(\sum_{\ell=1}^{\lfloor s\rfloor}\lambda_\ell((A_\jjj)_{\jjj \in \J_0}; \nu_0)-\varepsilon\right)\\\nonumber
&\geq m_0\left( m_1\left(\sum_{\ell=1}^{\lfloor s\rfloor }\lambda_\ell((A_\iii)_{\iii \in \I}; \mu)-\varepsilon\right)-\varepsilon\right)\\
&\geq m_0m_1\left(\sum_{\ell=1}^{\lfloor s\rfloor}\lambda_\ell((A_\iii)_{\iii \in \I}; \mu)-2\varepsilon\right),
\end{align*}
and similarly
\[\sum_{\ell=1}^{\lceil s\rceil}\lambda_\ell((A_\jjj)_{\jjj \in \J}; \nu)\geq m_0m_1\left(\sum_{\ell=1}^{\lceil s\rceil}\lambda_\ell((A_\iii)_{\iii \in \I}; \mu)-2\varepsilon\right).\]
Taking a convex combination of the two inequalities above it follows that the quantity
\begin{equation}\label{eq:that-thing-we're-interested-in}\sum_{\ell=1}^{\lfloor s\rfloor} \lambda_\ell((A_\jjj)_{\jjj \in \J}; \nu) + (s-\lfloor s\rfloor)\lambda_{\lceil s\rceil}((A_\jjj)_{\jjj \in \J}; \nu)\end{equation}
is bounded below by
\[m_0m_1\left( \sum_{\ell=1}^{\lfloor s\rfloor} \lambda_\ell((A_\iii); \mu) + (s-\lfloor s\rfloor)\lambda_{\lceil s\rceil}((A_\iii); \mu) -2\varepsilon\right).\]
Combining the lower bound \eqref{eq:entryism} on the entropy of $\nu$ with the preceding lower bound on \eqref{eq:that-thing-we're-interested-in} yields
\begin{align*}
 h(\nu) + \Lambda_s\left((A_\jjj)_{\jjj \in \J}; \nu\right) 
&\geq m_0m_1\left(h(\mu) + \Lambda_s((A_i)_{i \in \I}; \mu) - 4\varepsilon\right)\\
&=m_0m_1\left(P\left((A_i)_{i\in\I}; s\right) - 4\varepsilon\right)>0
\end{align*}
where we have used the fact that $\mu$ is a $\varphi^s$-equilibrium state for $(A_i)_{i \in \I}$. It follows that $\dimlyap ((A_\jjj)_{\jjj \in \J}; \mu) >s>\dimaff(A_i)_{i \in \I}-\delta$ as required to establish \eqref{it:uniform}. The proof is complete.

\section{Acknowledgements}

The research of I.D. Morris was partially supported by the Leverhulme Trust (Research Project Grant RPG-2016-194). C. Sert's research was supported by SNF grants 178958, 182089, SNF Ambizione 193481, the University of Zurich and the University of Warwick. I.D. Morris thanks ETH Zurich for its hospitality during the initial visit at which this project commenced. The authors thank Jialun Li for a very careful reading of an earlier version of this manuscript, and also thank Ariel Rapaport for additional helpful comments. 


\end{document}